\documentclass[a4paper,10pt,leqno,twoside]{amsart}

\usepackage[english]{babel} \usepackage{amsmath, amssymb, latexsym,
  epic, epsfig, rotating, fancyhdr, amsthm, pifont, enumitem}%, MnSymbol}
\usepackage{amscd,mathrsfs}
\usepackage[ansinew]{inputenc}
\usepackage{graphicx,psfrag}

\newcommand{\eqand}{\ensuremath{\quad \textrm{and} \quad}}

\newcommand{\foot}{\footnote}

\newcommand{\equi}{\ensuremath{\Leftrightarrow}}
\newcommand{\follows}{\ensuremath{\Rightarrow}}
\newcommand{\selfmap}{\ensuremath{\rcirclearrowleft}}

\newcommand{\ld}{\ensuremath{,\ldots,}}
\newcommand{\ssq}{\ensuremath{\subseteq}}
\newcommand{\smin}{\ensuremath{\setminus}}
\newcommand{\eps}{\ensuremath{\varepsilon}}
\newcommand{\wh}{\ensuremath{\widehat}}

\newcommand{\inte}{\ensuremath{\mathrm{int}}}

\newcommand{\diam}{\ensuremath{\mathrm{diam}}}

\newcommand{\kreis}{\ensuremath{\mathbb{T}^{1}}}

\newcommand{\Seins}{\ensuremath{S^{1}}}

\newcommand{\torus}{\ensuremath{\mathbb{T}^2}}

\newcommand{\homeo}{\ensuremath{\mathrm{Homeo}}}
\newcommand{\homtwo}{\ensuremath{\mathrm{Homeo}_0(\mathbb{T}^2)}}

\newcommand{\nfolge}[1]{\ensuremath{(#1)_{n\in\mathbb{N}}}}

\newcommand{\ifolge}[1]{\ensuremath{(#1)_{i\in\mathbb{N}}}}

\newcommand{\alphlist}{\begin{list}{(\alph{enumi})}{\usecounter{enumi}\setlength{\parsep}{2pt}
      \setlength{\itemsep}{1pt} \setlength{\topsep}{5pt}
      \setlength{\partopsep}{3pt}}}
\newcommand{\arablist}{\begin{list}{(\arabic{enumi})}{\usecounter{enumi}\setlength{\parsep}{2pt}
          \setlength{\itemsep}{1pt} \setlength{\topsep}{5pt}
          \setlength{\partopsep}{3pt}}}
\newcommand{\romanlist}{\begin{list}{(\roman{enumi})}{\usecounter{enumi}\setlength{\parsep}{2pt}
              \setlength{\itemsep}{1pt} \setlength{\topsep}{5pt}
              \setlength{\partopsep}{3pt}}}
\newcommand{\Romanlist}{\begin{list}{(\Roman{enumi})}{\usecounter{enumi}\setlength{\parsep}{2pt}
              \setlength{\itemsep}{1pt} \setlength{\topsep}{5pt}
              \setlength{\partopsep}{3pt}}}
\newcommand{\bulletlist}{\begin{list}{$\bullet$}{\setlength{\parsep}{2pt}
                \setlength{\itemsep}{1pt} \setlength{\topsep}{5pt}
                \setlength{\partopsep}{3pt}\setlength{\leftmargin}{15pt}}} 
\newcommand{\Alphlist}{\begin{list}{(\Alph{enumi})}{\usecounter{enumi}\setlength{\parsep}{2pt}
      \setlength{\itemsep}{1pt} \setlength{\topsep}{5pt}
      \setlength{\partopsep}{3pt}}}
 \newcommand{\listend}{\end{list}}

\newcommand{\T}{\ensuremath{\mathbb{T}}}

\newcommand{\N}{\ensuremath{\mathbb{N}}} 
\newcommand{\R}{\ensuremath{\mathbb{R}}}
\newcommand{\Z}{\ensuremath{\mathbb{Z}}}
\newcommand{\Q}{\ensuremath{\mathbb{Q}}}

\newcommand{\sph}{\ensuremath{\mathbb{S}}}
\newcommand{\A}{\ensuremath{\mathbb{A}}}

\newcommand{\cA}{\mathcal{A}}

\newcommand{\cC}{\mathcal{C}}
\newcommand{\cD}{\mathcal{D}}
\newcommand{\cE}{\mathcal{E}}

\newcommand{\cH}{\mathcal{H}}

\newcommand{\cL}{\mathcal{L}}

\newcommand{\cR}{\mathcal{R}}

\newcommand{\cT}{\mathcal{T}}
\newcommand{\cU}{\mathcal{U}}

\newcommand{\ncup}{\ensuremath{\bigcup_{n\in\N}}}

\newcommand{\icap}{\ensuremath{\bigcap_{i\in\N}}}

\newcommand{\ncap}{\ensuremath{\bigcap_{n\in\N}}}

\newcommand{\nLim}{\ensuremath{\lim_{n\rightarrow\infty}}}
\newcommand{\iLim}{\ensuremath{\lim_{i\rightarrow\infty}}}

\newcommand{\inergsum}{\ensuremath{\sum_{i=0}^{n-1}}}

\newcommand{\ntel}{\ensuremath{\frac{1}{n}}}

\newtheoremstyle{tobthm}{3pt}{3pt}{\itshape}{0pt}{\bfseries}{.}{0.5eM}{}
\theoremstyle{tobthm}

\newtheorem{definition}{Definition}[section]
\newtheorem{thm}[definition]{Theorem}

\newtheorem{lem}[definition]{Lemma}
\newtheorem{lemma}[definition]{Lemma}
\newtheorem{cor}[definition]{Corollary}
\newtheorem{prop}[definition]{Proposition}

\newtheorem{mainthm}{Theorem}

\newtheoremstyle{tobrem}{3pt}{3pt}{\normalfont}{0pt}{\bfseries}{.}{0.5em}{}
\theoremstyle{tobrem} 

\newtheorem{rem}[definition]{Remark}

\newtheorem{question}[definition]{Question}

\numberwithin{equation}{section}
\numberwithin{figure}{section}

\title{\Large\textsc{On torus homeomorphisms semiconjugate to
    irrational rotations}}

\thanks{Departament of mahtematics, TU-Dresden. Email addresses: {\tt Tobias.Oertel-Jaeger@tu-dresden.de}, {\tt alepasseggi@gmail.com}.}

\author{T.~J\"ager \and A.~Passeggi}

\begin{document}

\begin{abstract} 
  In the context of the Franks-Misiurewicz Conjecture, we study homeomorphisms of
  the two-torus semiconjugate to an irrational rotation of the circle. As a
  special case, this conjecture asserts uniqueness of the rotation vector in
  this class of systems. We first characterise these maps by the existence of an
  invariant `foliation' by essential annular continua (essential subcontinua of
  the torus whose complement is an open annulus) which are permuted with
  irrational combinatorics. This result places the considered class close to
  skew products over irrational rotations.  Generalising a well-known result of
  M.~Herman on forced circle homeomorphisms, we provide a criterion, in terms of
  topological properties of the annular continua, for the uniqueness of the
  rotation vector.
 
  As a byproduct, we obtain a simple proof for the uniqueness of the rotation
  vector on decomposable invariant annular continua with empty interior. In
  addition, we collect a number of observations on the topology and rotation
  intervals of invariant annular continua with empty interior.  \medskip

\end{abstract}

\maketitle

\section{Introduction}

Rotation Theory, as a branch of dynamical systems, goes back to Poincar\'e's
celebrated classification theorem for circle homeomorphisms. It states that
given an orientation-preserving circle homeomorphism $f$ with lift $F:\R\to\R$,
the limit
\[
\rho(F) \ = \ \nLim (F^n(x)-x)/n \ ,
\]
called the {\em rotation number} of $f$, exists and is independent of
$x$. Furthermore, $\rho(F)$ is rational if and only if there exists a
periodic orbit and $\rho(F)$ is irrational if and only if $f$ is
semiconjugate to an irrational rotation. 

Since both cases of the above dichotomy are easy to analyse, this result
provides a complete description of the possible long-term behaviour for a whole
class of systems without any additional {\em a priori} assumptions -- a
situation which is still rare even nowadays in the theory of dynamical
systems. In addition, the rotation number can be viewed as an element of the
first homological group of the circle and thus provides a link between the
dynamical behaviour of homeomorphisms and the topological structure of the
manifold. It is not surprising that the consequences of this result have found
numerous applications in the sciences, ranging from quantum physics to neural
biology \cite{johnson/moser:1982,perkeletal:1964}.  Hence, the attempt to apply
this approach to higher-dimensional manifolds, in order to obtain a
classification of possible dynamics in terms of rotation vectors and rotation
sets, is most natural. However, despite impressive contributions over the last
decades, fundamental problems still remain open even in dimension two.

Already in the case of the two-dimensional torus $\torus=\R^2/\Z^2$, a unique
rotation vector does not have to exist. Instead, given a torus homeomorphism $f$
homotopic to the identity and a lift $F:\R^2\to\R^2$, the {\em rotation set } is
defined as
\begin{equation*}
  \label{eq:1}
  \rho(F) \ = \ \left\{\rho\in\R^2\ \left|\ \exists z_i\in\R^2,\ n_i\nearrow\infty: 
    \iLim \left(F^{n_i}(z_i) - z_i\right)/n_i =\rho  \right.\right\} . 
\end{equation*}
This is always a compact and convex subset of the plane
\cite{misiurewicz/ziemian:1989}. Consequently, three principal cases can be
distinguished according to whether the rotation set (1) has non-empty interior,
(2) is a line segment of positive length or (3) is a singleton, that is, $f$ has
a unique rotation vector. Existing results on each of the three cases suggest
that a classification approach is indeed feasible: for example, in case (1) the
dynamics are `rich and chaotic', in the sense that the topological entropy is
positive \cite{llibre/mackay:1991} and all the rational rotation vectors in 
the interior of $\rho(F)$ are realised by periodic orbits \cite{franks:1989}; 
in case (3) a Poincar\'e-like classification exists under the additional assumption
of area-preservation and a certain {\em bounded mean motion} property \cite{jaeger:2009b},
 and the consequences of {\em unbounded mean motion} are being explored recently as well
\cite{jaeger:2009c,KoropeckiTal2012Irrotational,KoropeckiTal2012BoundedandUnbounded}.
In case the rotation set is a segment of positive length, examples can be
constructed whose rotation set is either (a) a segment with rational slope and
infinitely many rational points or (b) a segment with irrational slope and one
rational endpoint \cite{franks/misiurewicz:1990}. Recent results on torus
homeomorphisms with this type of rotation segments indicate that these examples
can be seen as good models for the general case
\cite{Davalos2012THwithRotationSegments,GuelmanKoropeckiTal2012Annularity}. In
addition, there exist many further results that provide more information on each
of the three cases. Just to mention some of the important contributions in this
direction, we refer to
\cite{misiurewicz/ziemian:1991,lecalvez:1991,kwapisz:2003%
,LeCalvez2005FoliatedBrouwer,jaeger:2009e,KoropeckiTal2012StrictlyToral}.

In the light of these advances, it seems reasonable to say that the outline of a
complete classification emerges. Yet, there is still a major blank spot in the
current state of knowledge. It is not known whether any rotation segment other
than the two cases (a) and (b) mentioned above can occur. Actually, it was
conjectured by Franks and Misiurewicz in \cite{franks/misiurewicz:1990} that
this cannot happen.  However, while this conjecture has been in the focus of
attention for more than two decades, it has defied all experts and up to date
there are still only very partial results on the problem. A deeper reason for
this may lie in the fact that it concerns dynamics without any periodic points
-- the main task is to exclude the existence of rotation segments without
rational points\foot{Note that a periodic orbit always has a rational rotation
  vector.} -- and therefore many standard techniques in topological dynamics
based on the existence of periodic orbits fail to apply. Hence, it is likely
that proving (or disproving) the conjecture will require to obtain a much better
understanding of periodic-point free dynamics, which seems a worthy task in a
much broader context as well.

We believe that in this situation the systematic investigation of
suitable subclasses of periodic point free torus homeomorphisms is a
good way to obtain further insight. In fact, there are some
classes that have been studied intensively already. First, Franks and Misiurewicz
proved that the conjecture is true for time-one maps of flows \cite{franks/misiurewicz:1990}. 
Secondly, Kwapisz considered torus homeomorphisms that preserve the 
leaves of an irrational foliation and showed that the rotation set is either a
segment with a rational endpoint or a singleton \cite{kwapisz:2000}.
Finally, for skew products over irrational rotations on the torus,
Herman proved the uniqueness of the rotation vector
\cite{herman:1983}. Hence, in both cases the conjecture was confirmed
for the particular subclasses, which are certainly very restrictive
compared to general torus homeomorphisms.  However, since these are
the only existing partial results on the problem, they are the only
obvious starting point for further investigations. The aim of this
article is to make a first step in this direction by studying torus
homeomorphisms which are semiconjugate to a one-dimensional irrational
rotation. For obvious reasons these do not have any periodic orbits,
but apart from this little is known about the dynamical implications
of this property. We first provide an analogous characterisation of
these systems.

Denote by $\homeo_0(\T^d)$ the set of homeomorphisms of the $d$-dimensional
torus that are homotopic to the identity. Recall that an {\em essential annular
  continuum} $A\ssq \T^2$ is a continuum whose complement $\torus\smin A$ is
homeomorphic to the open annulus $\A=\kreis\times\R$.  An {\em essential
  circloid} is an essential annular continuum which is minimal with respect to
inclusion amongst all essential annular continua. We refer to
Section~\ref{Preliminaries} for the corresponding definitions in higher
dimensions.  Note that for any family of pairwise disjoint essential continua in
$\T^d$ there exists a natural circular order. We say a wandering\foot{We call
  $A\ssq\T^d$ {\em wandering}, if $f^n(A)\cap A=\emptyset \ \forall n\geq 1$.}
essential continuum has irrational combinatorics (with respect to
$f\in\homeo_0(\T^d)$) if its orbit is ordered in $\T^d$ in the same way as the
orbit of an irrational rotation on $\kreis$. See Section~\ref{Semiconjugacy} for
more details.
\begin{mainthm}\label{t.semiconjugacy}
  Suppose $f\in\homeo_0(\T^d)$. Then the following statements are equivalent.
  \romanlist
  \item $f$ is semiconjugate to an irrational rotation $R$ of the
  circle;
\item there exists a wandering essential circloid with irrational
  combinatorics;
\item there exists a wandering essential continuum with irrational
  combinatorics;
  \item there exists a semiconjugacy $h$ from $f$ to $R$ such that for all
    $\xi\in\kreis$ the fibre $h^{-1}\{\xi\}$ is an essential annular
    continuum.\listend
\end{mainthm}

The proof is given in Section~\ref{Semiconjugacy}. Issues concerning
the uniqueness of the semiconjugacy in the above situation are
discussed in Section~\ref{UniqueSemiconjugacy}. In general the
semiconjugacy is not unique, but there exist important situations
where it is unique up to post-composition with a rotation. In this
case every semiconjugacy has only essential annular continua as
fibres.  

For the two-dimensional case, the implication ``(iii) \follows (i)'' in
Theorem~\ref{t.semiconjugacy} is contained in \cite{koropecki:2007}, and the
proof easily extends to higher dimensions. In our context, the most important
fact will be the equivalence ``(i) $\equi$ (iv)'', which says that the
semiconjugacy can always be chosen such that its fibres are annular
continua. This places the considered systems very close to skew products over
irrational rotations, with the only difference that the topological structure of
the fibres can be more complicated. For this reason, one may hope to generalise
Herman's result to this larger class of systems, thus proving the existence of a
unique rotation vector. To that end, however, we here have to make an additional
assumption on topological regularity the fibres of the semiconjugacy.

An essential annular continuum $A\ssq \torus$ admits essential simple closed
curves in its complement. The homotopy type of such curves is unique, and we
define it to be the {\em homotopy type of $A$}. We say $A$ is {\em horizontal}
if its homotopy type is $(1,0)$. Given a horizontal annular continuum $A$, we
denote by $\wh A$ a connected component of $\pi^{-1}(A)$, where
$\pi:\R^2\to\torus$ is the canonical projection. Let $T:\R^2\to\R^2,\
(x,y)\mapsto (x+1,y)$.  Then we say $A$ is {\em compactly generated} if there
exists a compact connected set $G_0\ssq \wh A$ such that $A=\bigcup_{n\in\Z}
T^n(G_0)$. An essential annular continuum with arbitrary homotopy type is said
to be compactly generated if there exists a homeomorphism of $\torus$ which maps
it to a compactly generated horizontal one.
%Note that while this is a rather weak regularity property, not all annular
%continua are compactly generated. In particular, this is not true for the
%pseudocircle constructed in \cite{Bing1951Pseudocircle} (see also
%\cite{handel:1982,herman:1986}).

\begin{mainthm} \label{t.herman} Suppose $f\in\homtwo$ is
  semiconjugate to an irrational rotation of the circle. Further,
  assume that the semiconjugacy $h$ is chosen such that its fibres
  $h^{-1}(\xi)$ are all essential annular continua and there exists a
  set $\Omega\ssq \kreis$ of positive Lebesgue measure such that
  $h^{-1}\{\xi\}$ is compactly generated for all $\xi\in\Omega$. Then
  $f$ has a unique rotation vector.
\end{mainthm}

The proof is given in Section \ref{Herman}.  We say an annular
continuum is {\em thin}, if it has empty interior. Note that in the
situation of Theorem~\ref{t.herman}, all but at most countably many of
the fibres are thin in this sense.  

As a byproduct of our methods, we also obtain the uniqueness of the
rotation vector for invariant compactly generated thin annular
continua.
\begin{mainthm}\label{uniqrotnumgen}
  Suppose $f\in\homtwo$ and $A$ is a thin annular continuum that is compactly
  generated and $f$-invariant. Then $f_{|A}$ has a unique rotation vector, that
  is, there exists a vector $\rho\in\R^2$ such that $\nLim (F^n(z)-z/n)=\rho$ for
  all $z\in\R^2$ with $\pi(z)\in A$. Moreover, the convergence is uniform in $z$.
\end{mainthm}
For decomposable\foot{A continuum is called {\em decomposable}, if it can be
  written as the union of two strict subcontinua.} circloids the above statement
was proved before by Le Calvez \cite{LeCalvezUniqRotNumDecomp}, using
Caratheodory's prime ends, which is a classical approach to study the rotation
theory of continua
\cite{CaratheodoryPrimeEnd1,CaratheodoryPrimeEnd2,KoroLeCalvezNassIrrPE}.  It
should be noted that there exist important examples of invariant thin annular
continua which are not compactly generated.  One example is the Birkhoff
attractor \cite{LeCalvez1988BirkhoffAttractor}, which does not have a unique
rotation vector and therefore cannot have compact generator due to the above
statement. Another well-known example is the pseudo-circle, which was
constructed by Bing in \cite{Bing1951Pseudocircle} and latter shown to occur as
a minimal set of smooth surface diffeomorphisms
\cite{handel:1982,herman:1986}. Whether pseudo-circles admit dynamics with
non-unique rotation vectors is still open. 

We close by collecting some observations on the topology and dynamics of
invariant thin annular continua in Section \ref{TopClassTACandCommRotTh}. It is
known that any thin annular continuum $A$ contains a unique circloid $C_A$ (see
Lemma~\ref{l.core_circloid}).  We show that if $A$ is compactly generated, then
so is the circloid $C_A$.  Conversely, if $C_A$ is compactly generated then
either $A$ is compactly generated as well or $A$ contains at least one {\em
  infinite spike}, that is, an unbounded connected component of $A\smin C_A$.
Finally, reproducing some examples due to Walker \cite{WalkerPrimeEnd} we show
that thin annular continua can have any compact interval as rotation segment,
even in the absence of periodic orbits.

\bigskip

\noindent{\bf Acknowledgements.} This work was supported by an
Emmy-Noether grant of the German Research Council (DFG grant JA
1721/2-1). We would like to thank Patrice Le Calvez for helpful
comments and Henk Bruin for fruitful discussions leading to the
remarks on the uniqueness of the semiconjugacies in
Section~\ref{UniqueSemiconjugacy}.

\section{Notation and preliminaries} \label{Preliminaries}

The following notions are usually used in the study of dynamics on the
two-dimensional torus or annulus. For convenience, we stick to the
same terminology also in higher dimensions. We let $\kreis=\R/\Z$ and
denote by $\A^d=\T^{d-1}\times \R$ the $d$-dimensional annulus. If
$d=2$, we simply write $\A$ instead of $\A^2$. We will often
compactify $\A$ by adding two points $-\infty$ and $+\infty$, thus
making it a sphere. As long as no ambiguities can arise, we will
always denote canonical quotient maps like $\R\to \kreis,\ \R^d\to
\T^d,\ \R^d\to \A^d$ by $\pi$.  Likewise, on any product space $\pi_i$
denotes the projection to the $i$-th coordinate. We call a subset
$A\ssq \A^d$ or $A\ssq \R^d$ bounded from above (from below) if
$\pi_d(A)$ is bounded from above (from below).

We say a continuum (that is, a compact and connected set) $E\ssq \A^d$
is {\em essential} if $\A^d\smin E$ contains two unbounded connected
components. In this case, one of these components will be unbounded
above and bounded below, and we denote it by $\mathcal{U}^+(A)$ . The
second unbounded component will be bounded above and unbounded below,
and we denote it by $\mathcal{U}^-(A)$. $A$ is called an {\em
  essential annular continuum} if $\A^d\smin A= \cU^+(A)\cup
\cU^-(A)$.  Note that in dimension two, one can show by using the
Riemann Mapping Theorem that both unbounded components are
homeomorphic to $\A$ and $A$ is the intersection of a decreasing
sequence of topological annuli. This is not true anymore in higher
dimensions, but at least we have the following.

\begin{lem} \label{l.ac-intersection}
  If $\nfolge{A_n}$ is a decreasing sequence of essential annular continua, then
  $A=\ncap A_n$ is an essential annular continuum as well.
\end{lem}
\proof As a decreasing intersection of essential continua, $A$ is an essential
continuum. Further, we have that $\T^d\smin A= \ncup \T^d\smin A_n$ is the union
of the two sets
\[
U^+ \ = \ncup \cU^+(A_n) \eqand  U^- = \ncup \cU^-(A_n) \ .
\]As the union of an increasing sequence of open connected sets is connected, 
both these sets are connected. Hence, $\T^d\smin A$ consists of exactly two 
connected components $\cU^+(A)=U^+$ and $\cU^-(A)=U^-$, both of which are unbounded.  \qed\medskip

Given $S\ssq\R^d$, we say $S$ is {\em horizontal} if $\pi_d(S)$ is bounded and
$\R^d\smin S$ contains two different connected components $\mathcal{U}^+(S)$ and
$\mathcal{U}^-(S)$ whose image under $\pi_d$ is unbounded. Note that in this
case one of the two components, which we always denote by $\mathcal{U}^+(S)$, is
bounded below whereas the other component, denoted by $\mathcal{U}^-(S)$, is
bounded above.  Similarly, given any subset $B\ssq \A^d$ bounded above (below)
we denote by $\mathcal{U}^+(B)$ ($\cU^-(B)$) the unique connected component of
$\A^d\smin B$ which is unbounded above (below). The same notation is used on
$\R^d$.  A horizontal connected closed set $S$ is called a {\em horizontal
  strip}, if $\R^d\smin S=\mathcal{U}^+(S)\cup\mathcal{U}^-(S)$. Note that thus
the lift of an essential annular continuum $A\ssq\A^d$ to $\R^d$ is a horizontal
strip.

In any $d$-dimensional manifold $M$, we say $A$ is an {\em annular
  continuum} if it is contained in a topological annulus
$\mathcal{A}\simeq \A^d$ and it is an essential annular continuum in
the above sense when viewed as a subset of $\mathcal{A}$. In this
situation, we say $A$ is {\em essential} if essential loops in
$\mathcal{A}$ are also essential in $M$.  We call $C\ssq \A^d$ an {\em
  essential circloid} if it is an essential annular continuum and does
not contain any other essential annular continuum as a strict subset.
Circloids in general manifolds are then defined in the same way as
annular continua. Finally, we call a horizontal strip $S$ {\em
  minimal} if it is a minimal element of the set of horizontal strips
with the partial ordering by inclusion. Note that thus an annular
essential continuum in $\A^d$ is a circloid if and only if its lift to
$\R^d$ is a minimal strip. Finally, we call a closed set {\em thin},
if it has empty interior.

\begin{lem}[\cite{jaeger:2009b}, Lemma 3.4] \label{l.core_circloid}
  Every thin annular continuum $A$ contains a unique circloid $C_A$,
  which is given by
  \begin{equation}
    \label{e.core-circloid}
   C_A \ = \ \overline{\mathcal{U}^+(A)} \cap \overline{\mathcal{U}^-(A)} \ .
  \end{equation}
  The same statement applies to thin horizontal strips.
\end{lem}
The proof in \cite{jaeger:2009b} is given for essential annular
continua and for $d=2$, but it literally goes through in higher
dimensions and for strips. The same is true for the following result,
which describes an explicit construction to obtain essential circloids
from arbitrary essential continua. Given an essential set $A\ssq \R^d$
which is bounded above, we write $\cU^{+-}(A)$ instead of
$\cU^-(\cU^+(A))$ and use analogous notation for other concatenations
of these procedures.
\begin{lem}[\cite{jaeger:2009b}, Lemma 3.2]
  If $A\ssq \A^d$ is an essential continuum, then
\[
\cC^+(A)  =  \T^d\smin \left(\cU^{+-}(A)\cup\cU^{+-+}(A)\right) \eqand \cC^-(A) = \T^d\smin
\left(\cU^{-+}(A)\cup \cU^{-+-}(A)\right)
\]
are circloids. Further, we have $\partial \cC^\pm(A) \ssq A$.
\end{lem}

Given two horizontal essential continua $A_1,A_2\ssq \T^d$, we say $\hat A_i\ssq
\A^d$ is a lift of $A_i$ if it is a connected component of $\pi^{-1}(A_i)$. We
write $\hat A_1 \preccurlyeq \hat A_2$ if $\hat A_2\ssq \cU^+(\hat A_1)$. Given
two lifts $\hat A_1\preccurlyeq \hat A_2$ we say they are {\em adjacent}, if if
the compact region $\A^d\smin\left(\cU^-(\hat A_1)\cup \cU^+(\hat A_2)\right)$
does not contain any integer translates of $\hat A_1$ and $\hat A_2$. In this
situation, we let $\left[\hat A_1,\hat A_2\right] = \A^d\smin\left(\cU^-(\hat
  A_1)\cup \cU^+(\hat A_2)\right)$ and $\left(\hat A_1,\hat A_2\right) =
\cU^+(\hat A_1) \cap \cU^-(\hat A_2)$. Then we let $\left[
  A_1,A_2\right]=\pi\left(\left[\hat A_1,\hat A_2\right]\right)$ and
$\left(A_1,A_2\right)=\pi\left(\left(\hat A_1,\hat A_2\right)\right)$.  With
these notions, we define a circular order on pairwise disjoint essential
continua $A_1,A_2,A_3\ssq \T^d$ by
\[
A_1 \preccurlyeq A_2 \preccurlyeq A_3 \quad \equi \quad A_2\in
\left[A_1,A_3\right] \ .
\]
The strict relation $\prec$ is defined by replacing closed \textit{intervals} with open
ones.  Using these notions, we now say a sequence \nfolge{A_n} of pairwise
disjoint essential continua in $\T^d$ has {\em irrational combinatorics} if
there exists $\rho\in\R\smin\Q$ such that for arbitrary $y_0\in\kreis$ the
sequence $y_n=y_0+n\rho\bmod 1$ satisfies
\[
  A_k \prec A_m \prec A_n \quad \equi \quad  y_k< y_m < y_n
\]
for all $k,m,n\in\Z$. 

We finish the topological preliminaries with the following proposition. 

\begin{prop}\label{corjor}
  Let $A,B$ be compact subsets of $\A^n$ such that:
\begin{itemize}
 \item $A\cap B=\emptyset$
 \item $\A^n\setminus A$ has exactly one unbounded component.
 \item $\A^n\setminus B$ has exactly one unbounded component.
\end{itemize}
Then $\A^n\setminus A\cup B$ has one exactly unbounded component.
\end{prop}

\begin{proof}
 Let $\Pi_{0,b}(A^c),\Pi_{0,b}(B^c)$ be the sets of bounded connected 
components of $A^c,B^c$ respectively. We define

\begin{center} $\textrm{Fill}(A)=A\cup\left[ \bigcup_{U\in\Pi_{0,b}(A^c)}U\right]$ , 
$\textrm{Fill}(B)=B\cup\left[ \bigcup_{U\in\Pi_{0,b}(B^c)}U\right]$
 \end{center}

 Hence $A\subset \textrm{Fill}(A)$, $B\subset\textrm{Fill}(B)$ and
 $\textrm{Fill}(A)^c ,\textrm{Fill}(B)^c$ have each only one connected component
 which is unbounded.  Furthermore, we may choose sets $C,D\ssq\A^n$ such that
 $\textrm{Fill}(A)\cup\textrm{Fill}(B)=C\cup D$ and we have that 
\begin{itemize}
\item $C,D$ are compact;
\item $C\cap D=\emptyset$;
\item $C^c,D^c$ has each only one connected component.
\end{itemize}
Notice that one of the two sets $C$ or $D$ could be empty.

Now, it suffices to show that $(C\cup D)^c$ has only one connected component.
For this we consider the usual compactification of the annulus by the sphere
$\sph^n$, and the unreduced Mayer-Vietoris long sequence (see \cite{Hatcher2002})
for $\left(\sph^n,C^c,D^c\right)$ starting from $H_1(\sph^n)$:

\begin{center}
$H_1(\sph^n)\overset{\partial_*}{\rightarrow}H_0(C^c\cap D^c)
\overset{\theta_*}{\rightarrow}H_0(C^c)\oplus H_0(D^c)
\overset{\xi_*}{\rightarrow}H_0(\sph^n)\rightarrow 0$
\end{center}

Then, we obtain that 

\begin{center}
$0\overset{\partial_*}{\rightarrow}H_0(C^c\cap D^c)
\overset{\theta_*}{\rightarrow}\Z\oplus \Z
\overset{\xi_*}{\rightarrow}\Z\rightarrow 0$
\end{center}

Thus we have that $H_0(C^c\cap D^c)$ is isomorphic
to $\textrm{ker}(\xi_*)$, so $H_0(C^c\cap D^c)\cong\Z$.
This implies that there is only one connected component in 
$C^c\cap D^c$. 
\end{proof}

Finally, we will frequently use the following Uniform Ergodic Theorem
(e.g.\ \cite{katok/hasselblatt:1997,sturman/stark:2000}).
\begin{thm} \label{t.uniform_ET}
  Suppose $X$ is a compact metric space and $f:X\to X$ and $\varphi:
  X\to \R$ are continuous. Further, assume that there exists
  $\rho\in\R$ such that 
  \[
        \int_X \varphi \ d\mu \ = \ \rho
  \]
  for all $f$-invariant ergodic probability measures $\mu$ on $X$.
  Then 
  \[ 
        \nLim \frac{1}{n}\left(\inergsum \varphi\circ f^i(x)\right) \ = \ \rho \quad \textrm{for all } x\in X\ .
  \]
  Furthermore, the convergence is uniform on $X$.
\end{thm}

\section{Semiconjugacy to an irrational rotation} \label{Semiconjugacy}

We now turn to the proof of Theorem~\ref{t.semiconjugacy}. The
implications (ii)$\follows$(iii) and (iv)$\follows$(i) in
Theorem~\ref{t.semiconjugacy} are obvious. Hence, in order to prove
all the equivalences, it suffices to prove (iii)$\follows$(ii),
(i)$\follows$(iii) and (ii)$\follows$(iv). We do so in three separate
lemmas and start by treating the easiest of the three
implications, which is (iii)$\follows$(ii).

\begin{lem}
  Let $f\in\homeo_0(\T^d)$ and suppose $E$ is a wandering essential
  continuum. Then $\cC^+(E)$ is a wandering essential circloid and the circular
  ordering of the orbits of $E$ and $\cC^+(E)$ are the same.
\end{lem}
\proof Suppose $f\in\homeo_0(\T^d)$ and $E$ is a wandering essential continuum
with irrational combinatorics. Let $E_n=f^n(E)$ and
$C_n=\cC^+(E_n)=f^n(\cC^+(E))$. Assume for a contradiction that the $C_n$ are
not pairwise disjoint, that is, $C_i\cap C_j\neq \emptyset$ for some integers
$i\neq j$. Since $\partial C_n\ssq E_n$ for all $n\in\Z$, $C_i$ must intersect
the interior of $C_j$ or vice versa. Assuming the first case, by connectedness
$C_i$ has to be contained in a single connected component of $\T^d\smin \partial
C_j$, since otherwise it would have to intersect $\partial C_j$. As $C_i$ is not
contained in $\T^d\smin C_j$, it has to be contained in the interior of
$C_j$. However, this contradicts the minimality of circloids with respect to
inclusion. Hence $C_0$ is wandering. 

The circular ordering is preserved when going from $\nfolge{E_n}$ to \nfolge{C_n} is
now obvious.
\qed\medskip

The next lemma shows (ii)$\follows$(iv).

\begin{lem}
  Let $f\in\homeo_0(\T^d)$ and suppose $C$ is a wandering essential
  circloid with irrational combinatorics of type $\rho$. Then there
  exists a semiconjugacy $h :\T^d\to\kreis$ to $R_\rho$ such that the
  fibres $h^{-1}\{\xi\}$ are all essential annular continua.
\end{lem}
\proof We let $C_n:= f^n(C)$ and denote the connected components of
the lifts of these circloid by $\hat C_{n,m}$, where the indices are
chosen such that for all integers $n,m$ we have
\begin{itemize}
\item $\pi(\hat C_{n,m}) = C_n$;
\item $F(\hat C_{n,m}) = \hat C_{n+1,m}$;
\item $T(\hat C_{n,m}) = \hat C_{n,m+1}$. 
\end{itemize}
We claim that 
\[
H(z) \ = \ \sup\left\{n\rho+m\mid z\in\cU^+(\hat C_{n,m})\right\} \ 
\]
is a lift of a semiconjugacy $h$ with the required properties.  Note
that due to the irrational combinatorics we have $n\rho+m \leq \tilde
n\rho+\tilde m$ if and only if $\hat C_{n,m} \preccurlyeq \hat C_{\tilde n,\tilde
  m}$, such that in particular $H(z)$ is well-defined and finite for
all $z\in\A^d$. Further, for any $z\in\A^d$ we have
\begin{eqnarray*}
  H\circ F(z) & = & \sup\left\{n\rho+m\mid F(z)\in\cU^+(\hat C_{n,m})\right\} \\
  & = & \sup\left\{n\rho+m\mid z\in\cU^+(\hat C_{n-1,m})\right\} \ = \ H(z) + \rho \ .
\end{eqnarray*}
In a similar way one can see that $H\circ T(z) = H(z)+1$, such that
$H$ projects to a map $h:\T^d\to \kreis$ which satisfies $h\circ
f=R_\rho\circ h$.

In order to check the continuity of $H$, suppose $U\ssq\R$ is an open interval
and let $z\in H^{-1}(U)$. Choose $r=n\rho+m<H(z)<\tilde n\rho+\tilde m=s$ with
$r,s\in U$. Then $z\in \cU^+(C_{n,m}) \cap \cU^-(C_{\tilde m,\tilde n})
=:V$. From the definition of $H$ we see that $H(V)\ssq [r,s] \ssq U$, and thus
$H^{-1}(U)$ contains an open neighbourhood of $z$. Since $U$ and $z\in H^{-1}(U)$
were arbitrary, $H$ is continuous.  The fact that $h$ is onto follows easily
from the minimality of $R_\rho$, so that $h$ is indeed a semiconjugacy from $f$
to $R_\rho$.

It remains to prove the fact that the fibres $h^{-1}\{\xi\}$ are
annular continua. In order to do so, note that for $\xi\in\kreis$
\begin{equation}
  \label{eq:2}
\begin{split}
  H^{-1}\{\xi\} & \ = \ \bigcap_{n\rho+m<\xi} \cU^+(\hat C_{n,m}) \ \cap
  \bigcap_{\tilde n+\rho\tilde m >\xi} \cU^-(\hat C_{\tilde n,\tilde m}) \\
  & \ = \ \bigcap_{n\rho+m<\xi} \A^d \smin \cU^-(\hat C_{n,m}) \ \cap
  \bigcap_{\tilde n\rho+\tilde m >\xi} \A^d \smin \cU^+(\hat C_{\tilde
    n,\tilde m}) \ .
\end{split}
\end{equation}
Note here that for all $n,m,n',m'$ with $n\rho+m<n'\rho+m'$ we have
\[
\cU^+(\hat C_{n',m'}) \ \ssq \ \A^d\smin \cU^-(\hat C_{n',m'}) \ \ssq \ 
\cU^+(\hat C_{n,m}) 
\]
and similar inclusions hold in the opposite direction. This explains

the second equality in (\ref{eq:2}).  Choosing sequences
$n_i,m_i,\tilde n_i,\tilde m_i$ with $n_i\rho+m_i\nearrow \xi$ and
$\tilde n_i\rho + \tilde m_i\searrow \xi$, we can rewrite (\ref{eq:2})
as
\[
H^{-1}\{\xi\} \ = \ \icap\A^d \smin \left(\cU^-(\hat C_{n_i,m_i}) \cup
  \cU^+(\hat C_{\tilde n_i,\tilde m_i})\right) \ .
\]
Since the sets of the intersection are all essential annular continua,
so is $H^{-1}\{\xi\}$ by Lemma~\ref{l.ac-intersection}.  \qed\medskip

It remains to prove the implication (i)$\follows$(iii).
\begin{lemma}
  Suppose $h:\T^d\to\kreis$ is a semiconjugacy from $f\in\homeo_0(\T^d)$ to an
  irrational rotation $R_\rho$. Then every fibre $h^{-1}\{\xi\}$ contains a
  wandering essential continuum with irrational combinatorics.
\end{lemma}

\begin{proof}
  We first show that the action $h^*:\Pi_1(\T^d)\to\Pi_1(\kreis)$ of
  $h$ on the fundamental groups is non-trivial. Suppose for a
  contradiction that $h^*=0$. Then any lift $H:\R^d\to\R$ of $h$ is
  bounded since in this case $\sup_{z\in\R^d}\|H(z)\| = \sup_{z\in
    [0,1]^d}\|H(z)\|$. However, this contradicts the unboundedness of
\[
  H\circ F^n(z) \ = \ R_\rho^n\circ H(z) \ .
\]
Consequently, $h^*$ is non-trivial, and by composing $h$ with a linear
torus automorphism we may assume that $h^*$ is just the projection to
the last coordinate. This composition may change the rotation number,
but does not effect its irrationality. We obtain a lift $\hat
h:\A^d\to\R$ which satisfies $\hat h(z)\to\pm\infty$ if $z\to
\pm\infty$.

As a consequence, the Intermediate Value Theorem implies that every
properly embedded line $\Gamma=\{\gamma(t)\mid t\in\R\}$ intersects
all level sets $\wh E_x=\hat h^{-1}\{x\}$. Hence, all $\wh E_x$ are
essential.

If $\wh E_x$ is not connected, we consider the family of all compact
and essential subsets of $\wh E_x$ and choose and element $\wh \cE$
which is minimal with respect to the inclusion. Note that such minimal
elements exist by the Lemma of Zorn. By Proposition \ref{corjor} $\wh
\cE$ is connected. Further, $\cE=\pi(\wh\cE)$ is wandering since $\wh
\cE\ssq h^{-1}\{x\}$. Hence, $\cE$ is the wandering essential
continuum we are looking for. The fact that $\cE$ has irrational
combinatorics can be seen from the semi-conjugacy equation.
\end{proof}

\section{On the uniqueness of the semiconjugacy.}\label{UniqueSemiconjugacy}

In the light of the preceding section, it is an obvious question to
ask to what extent a semiconjugacy between $f\in\homeo(\T^2)$ and an
irrational rotation $R_\rho$ of the circle is unique. It is easy to
check that for every rigid rotation $R:\kreis\to\kreis$
the map $R\circ h$ is a semiconjugacy between $f$ and $R_{\rho}$ as
well. Hence there is non-uniqueness of the semiconjugacy in general.
Nevertheless, one could ask whether there is uniqueness up to
post-composition with rotations. In brief, we will speak of {\em
  uniqueness modulo rotations}.

Consider $f\in\homtwo$ given by $f(x,y)= (x+\rho_1,y)$ with
$\rho_1\in\Q^c$. For any continuous function $\alpha:\T^1\rightarrow\T^1$, we
have that $h_{\alpha}(x,y)=x+\alpha(y)$ is a semiconjugacy from $f$ to
$R_{\rho_1}$. Thus we do not have uniqueness of the
semiconjugacy even modulo rotations. However, it is not difficult to
see that all the possibles semiconjugacies between $f$ and
$R_{\rho_1}$ are given by $h_{\alpha}$ for some continuous function
$\alpha$. This implies in particular that on every minimal set
$Y_r=\{(x,y)\in\T^2|\mbox{ }y=r\}$, $r\in \T^1$, given any two
semiconjugacies $h_1$ and $h_2$ we have that $h_{1|Y_r}=\left(R\circ
  h_2\right)_{|Y_r}$ for some rigid rotation $R$.  This, as we will
see, is a general fact.

We say that an $f$-invariant set $\Omega$ is {\em externally
  transitive} if for every $x,y\in\Omega$ and neighbourhoods $U_x,U_y$
of $x$ and $y$, respectively, there exists $n\in\N$ such that
$f^n(U_x)\cap U_y\neq \emptyset$. Notice that $f^n(U_x)$ and $U_y$ do
not need to intersect in $\Omega$ as in the usual definition of
topological transitivity. In the above example the sets $Y_r$ are
transitive, hence externally transitive.

Given $f\in\homtwo$ semiconjugate to a rigid rotation $R_\rho$ and a 
$f$-invariant set $\Omega\ssq\torus$, we say the semiconjugacy is 
{\em unique modulo rotations on $\Omega$} if for all semiconjugacies 
$h_1,h_2$ from $f$ to $R_\rho$ we have $h_{1|\Omega}=\left(R\circ h_2\right)_{|\Omega}$
for some rigid rotation $R$.

\begin{prop}\label{uniqexttrans} 
  Let $f\in\homeo(\T^2)$ be semiconjugate to a rigid rotation of
  \kreis. Further, assume that $\Omega\subset \T^2$ is an externally
  transitive set of $f$. Then the semiconjugacy is unique modulo
  rotations on $\Omega$.
\end{prop}
\begin{proof}
  Let $h_1,h_2$ be two semiconjugacies between $f$ and $R_{\rho}$.  By
  post-composing with a rigid rotation, we may assume that
  $h_1(x)=h_2(x)$ for some $x\in\Omega$. Suppose for a contradiction
  that $h_1(y)\neq h_2(y)$ for some $y\in\Omega$.

  Let $\varepsilon=\frac{1}{2}\cdot d(h_1(y),h_2(y))$ and $\delta>0$ such that
  $d(h_1(x'),h_2(x'))<\varepsilon$ if $x'\in B_{\delta}(x)$ and
  $d(h_1(y'),h_2(y'))>\varepsilon$ if $y'\in B_{\delta}(y)$.  Due to $\Omega$
  being externally transitive, there exists $z\in B_{\delta}(x)$ and $n\in\N$
  such that $f^n(z)\in B_{\delta}(y)$.  However, at the same time we have that
  $\eps<d(h_1(f^n(z)),h_2(f^n(z)))=d(R^n_\rho(h_1(z)),R^n_\rho(h_2(z)))=d(h_1(z),h_2(z))<\eps$,
  which is absurd.
\end{proof}

As a consequence, we obtain the uniqueness of the semiconjugacy modulo
rotations whenever the non-wandering set of $f$ is externally
transitive. The reason is the following simple observation.
\begin{lemma}\label{aboutsemiconj}
If $h_1(x)=h_2(x)$ for two semiconjugacies between $f\in\homeo(\T^2)$ and a rigid rotation
of $\kreis$, then $h_1(y)=h_2(y)$ for all $y$ with $x\in\overline{\mathcal{O}(y,f)}$. 
\end{lemma}

\begin{proof} Suppose for a contradiction that $x\in\overline{\mathcal{O}(y,f)}$
  but $h_1(y)\neq h_2(y)$. Let 
\\$\varepsilon= d(h_1(y),h_2(y))/2$ and $\delta>0$
  such that if $x'\in B_{\delta}(x)$ then $h_1(x'),h_2(x')\in
  B_{\varepsilon}(h_1(x))$.  Further, let $n\in\N$ be such that $z:=f^n(y)\in
  B_{\delta}(x)$. Then on one hand $h_1(z),h_2(z)\in B_{\varepsilon}(h_1(x))$,
  and on the other hand $d(h_1(z),h_2(z))=d(h_1(y),h_2(y))=2\varepsilon$, which
  is absurd.
\end{proof}

Given $f\in\homeo(\T^2)$ we denote its non-wandering set by
$\Omega(f)$. Since any orbit accumulates in the non-wandering set, the
combination of Proposition~\ref{uniqexttrans} and
Lemma~\ref{aboutsemiconj} immediately yields
\begin{cor}\label{uniqcond}
  Suppose that $f\in\homeo(\T^2)$ is semiconjugate to a rigid rotation
  of $\kreis$.  Further assume that $\Omega(f)$ is externally
  transitive.  Then the semiconjugacy is unique modulo rotations.
\end{cor}
For irrational pseudorotations of the torus,\foot{That is, torus
  homeomorphisms homotopic to the identity with unique and totally
  irrational rotation vector.} external transitivity of the
non-wandering set was proved by R.~Potrie in
\cite{Potrie2012RecNonResTorusHomeo}. Hence, applying
Corollary~\ref{uniqcond} in both coordinates yields 
\begin{cor}\label{UniqforPsRot}
  Let $f\in\homeo(\T^2)$ be an irrational pseudo-rotation which
  is seminconjugate to the respective rigid translation of $\T^2$.
  Then the semiconjugacy is unique up to composing with rigid
  translations of $\T^2$.
\end{cor}
Finally, one may ask the following.
\begin{question}
  Does every semiconjugacy between $f\in\homtwo$ and a rigid rotation
  on \kreis\ have essential annular continua as fibres?
\end{question}
We note that in the example $f(x,y)=(x+\rho_1,y)$ discussed above this
is true, since the fibres of the semiconjugacy $h_\alpha$ are the
essential circles $\{(x-\alpha(y),y)\mid y\in\kreis\}$, $x\in\kreis$.
By Theorem~\ref{t.semiconjugacy} it is also true whenever the
semiconjugacy is unique modulo rotations, since there always exists
one semiconjugacy with this property and the topological structure of
the fibres is certainly preserved by post-composition with rotations.

\section{Fibred rotation number for foliations of circloids} \label{Herman}

The aim of this section is to prove Theorem~\ref{t.herman}. In order
to do so, we need some further preliminary results.  Given two open
connected subsets $U,V$ of a manifold $M$, we say that $K\ssq
M\smin(U\cup V)$ {\em separates} $U$ and $V$ if $U$ and $V$ are
contained in different connected components of $M\smin K$.
\begin{lem} \label{l.core_strip_lemma} Suppose $S \ssq \R^d$ is a thin
  horizontal strip and $K\ssq S$ is a connected closed set that
  separates $\mathcal{U}^+(S)$ and $\mathcal{U}^-(S)$.  Then $C_S\ssq
  K$.
\end{lem}
\proof Suppose $C_S \nsubseteq K$ and let $z\in C_S\smin K$. Then
$B_\eps(z)\ssq \R^d\smin K$. However, as $B_\eps(z)$ intersects both
$\mathcal{U}^+(S)$ and $\mathcal{U}^-(S)$ by
Lemma~\ref{l.core_circloid}, this means that $\mathcal{U}^+(S) \cup
B_\eps(z) \cup \mathcal{U}^-(S)$ is contained in a single connected
component of $\R^d\smin K$, contradicting the fact that $K$ separates
$\mathcal{U}^+(S)$ and $\mathcal{U}^-(S)$.  \qed\medskip

Given an essential thin annular continuum $A\ssq\A$, we denote its
lift to $\R^2$ by $\widehat A=\pi^{-1}(A)$. Let $T:\R^2\to\R^2,\
(x,y)\mapsto (x+1,y)$. Then we say $A$ has a {\em compact generator},
if there exists a compact connected set $G_0\ssq\wh A$ such that
$\bigcup_{n\in\Z} G_n=\wh A$, where $G_n=T^n(G_0)$. 
\begin{lem} \label{l.generator-intersection}
  If $A\ssq \A$ is a thin annular continuum with generator $G_0$, then
 $G_n\cap G_{n+1}\neq \emptyset$  for all $n\in\N$. 
\end{lem}
\proof It suffices to prove that $G_0\cap G_1\neq \emptyset$.  As $C_{\wh A} =
\bigcup_{n\in\Z} G_n\cap C_{\wh A}$ and $C_{\wh A}$ is connected, the compact
sets $G_n\cap C_{\wh A}$ cannot be pairwise disjoint. Hence, for some $k\in\N$
we have $G_0\cap G_k\cap C_{\wh A}\neq\emptyset$ and we can therefore choose a
point $z_0\in G_0\cap C_{\wh A}$ with $z_k=T^k(z_0)\in G_0\cap G_k$.  If $k=1$,
then $z_1=T(z_0)\in G_0\cap G_1$, and we are finished in this case.

Hence, suppose for a contradiction that $k>1$ and $z_1=T(z_0)\notin
G_0$. Then $B_\eps(z_1)\cap G_0=\emptyset$ for some $\eps>0$. Since
$B_\eps(z_1)$ intersects both $\mathcal{U}^+(\wh A)$ and
$\mathcal{U}^-(\wh A)$ by Lemma~\ref{l.core_circloid}, this means that
we can choose a proper curve $\Gamma\ssq \mathcal{U}^+(\wh A) \cup
B_\eps(z_1)\cup \mathcal{U}^-(\wh A)$ passing through $z_1$ such that
$\pi_1(\Gamma)$ is bounded and $\pi_2(\Gamma) =\R$. In addition, by
choosing $\Gamma$ as the lift of a proper curve $\gamma\ssq
\mathcal{U}^+(A) \cup B_\eps(\pi(z_1))\cup \mathcal{U}^-(A)\ssq \A$
joining $-\infty$ and $+\infty$, we can assume that $T(\Gamma)\cap
\Gamma\neq \emptyset$.

$\Gamma$ separates $\R^2$ into two open connected components $D^-$
unbounded to the left and $D^+$ unbounded to the right. As $T(\Gamma)$
is disjoint from $\Gamma$ we have $T(\overline{D^+})\ssq D^+$.  This
implies that $z_0\in D^-$, otherwise we would have $z_1=T(z_0)\in
T\left(\overline{D^+}\right)\ssq D^+$, contradicting $z_1\in\Gamma$.
At the same time, we have $z_k = T^{k-1}(z_1) \in
T^{k-1}\left(\overline{D^+}\right) \ssq D^+$.  However, this means that $z_0$ and
$z_k$ are in different connected components of $\R^2\smin \Gamma$,
contradicting the connectedness of $G_0$. \qed\medskip

Given any bounded
set $B\ssq \R^2$, we let
\[
\nu_B \ = \ \max\{ n\in\N\mid \exists z\in B: T^n(z)\in B\} \ .
\]

\begin{lem}\label{l.main_generator_lemma}
  Suppose $A,A'\ssq \A$ are thin essential annular continua with
  compact generators $G_0,G'_0$.  Further, assume $f\in\homeo_0(\A)$ maps
  $A$ to $A'$. Then for any lift $F$ of $f$ the set $F(G_0)$
  intersects at most $\nu_{G_0}+\nu_{G'_0}+1$ integer translates of
  $G_0'$.
\end{lem}
\proof Suppose $F(G_0)$ intersects $G_n'$ and $G_m'$ for some $m>n$. Then due to
Lemma~\ref{l.generator-intersection}, the set 
\[
\bigcup_{k\leq n} G_k' \cup F(G_0) \cup \bigcup_{k\geq m}G_k' \ \ssq \ \wh A
\]
is connected and therefore separates $\mathcal{U}^+(\wh A)$ and
$\mathcal{U}^-(\wh A)$. Hence, by Lemma~\ref{l.core_strip_lemma} it
contains $C_{\wh A}=\wh{C_A}$.  Let $z_0\in G_0'\cap C_{\wh A}$ and
assume without loss of generality that $z_j\notin G_0'$ for all $j\geq
1$. Then $z_n\in G_n'$ and $z_j\notin \bigcup_{k\leq n}G_k'$ for all
$j>n$. Furthermore, since $z_m\in G_m'$ we have that $z_j\notin
\bigcup_{k\geq m}G_k'$ for all $j<m-\nu_{G_0'}$. Thus, we must have
\[
\{ z_{n+1}\ld z_{m-\nu_{G_0'}-1}\} \ \ssq \ F(G_0) \ .
\]
However, since $F(G_0)$ contains at most $\nu_{G_0}$ integer
translates of $z_0$, this implies $m-n\leq \nu_{G_0}+\nu_{G_0'}+1$.
\qed\medskip

As a first consequence, this yields the following.
\begin{cor} \label{c.uniquness} Let $f\in\homeo_0(\A)$ with lift
  $F:\R^2\to\R^2$ and suppose $A$ is a thin essential annular
  continuum which is compactly generated. Then $f_{|A}$ has a unique
  rotation number, that is,
  \[
    \rho_A(F) \ = \ \nLim \pi_2\circ (F^n(z)-z)/n
  \]
  exists for all $z\in \pi^{-1}(A)$ and is independent of $z$. Moreover,
the convergence is uniform in $z$.
\end{cor}
\begin{proof}
  As $\rho(F,z) = \nLim \pi_2\circ (F^n(z)-z)/n = \nLim \ntel
  \inergsum \varphi\circ f^i(z)$ is an ergodic sum with observable
  $\varphi(z)=\pi_2(F(z)-z)$, we have that $\rho(f,z)=\int_A \varphi \
  d\mu =: \rho(\mu)$ $\mu$-a.s.\ for every $f$-invariant probability
  measure supported on $A$. Note here that $\varphi$ is well-defined
  as a function $\A\to\R$. Assume for a contradiction that the
  rotation number is not unique on $A$. Then
  Theorem~\ref{t.uniform_ET} implies the existence of two
  $f$-invariant ergodic measures $\mu_1,\mu_2$ supported on $A$ with
  $\rho(\mu_1)\neq \rho(\mu_2)$. Consequently, we can choose
  $z_1,z_2\in A$ with $\rho(F,z_1)=\rho(\mu_1) \neq
  \rho(F,z_2)=\rho(\mu_2)$. However, at the same time we may choose
  lifts $\hat z_1,\hat z_2\in G_A$ of $z_1,z_2$, where $G_A$ is a
  compact generator of $A$. Then Lemma~\ref{l.main_generator_lemma}
  implies that $F^n(\hat z_1)$ and $F^n(\hat z_2)$ are contained in
  the union of $2\nu_{G_A}+1$ adjacent copies of $G_A$. Consequently,
  we have that $d(F^n(\hat z_1),F^n(\hat z_2)) \leq
  \diam(G_A)+2\nu_{G_A}+1$ for all $n\in\N$, a contradiction. The
  uniform convergence follows from the same argument. 
\end{proof}

The following result is a complement of the corollary above.

\begin{cor}

Let $f\in\homeo_0(\A)$ with lift
$F:\R^2\to\R^2$. Further, suppose $A$ is a $f$-invariant thin essential annular
continuum which is compactly generated and $\rho_A(F)=\{0\}$.
Then $F$ has a fixed point in $\pi^{-1}(A)$.
\end{cor}
 
\begin{proof}
 
Let $G_0$ be a compact generator of $A$, and for every $k\in\Z$ the set $G_k:=T^k(G_0)$. 
We first claim that $F^n(G_0)\cap G_0\neq\emptyset$ for all $n\in\Z$. Otherwise we
have $n_0\in\N$ with either $F^{n_0}(G_0)\subset \bigcup_{k\geq 1}G_k$ or  
$F^{n_0}(G_0)\subset \bigcup_{k\leq -1}G_k$. Let us consider the first situation,
for the complementary one a symmetric argument gives the contradiction. Then, for
every $j\in\N$ we have that $F^{jn_0}(G_0)\subset \bigcup_{k\geq j}G_k$. This, however,
implies that the rotation number of $A$ is strictly positive, a contradiction. 

Consider now $C:=\overline{\left(\bigcup_{k\in\Z}F^k(G_0)\right)}$. Thus Lemma \ref{l.main_generator_lemma}
implies that $C$ is a compact and invariant set. Moreover, as $A$ is thin, $C$ is a non-separating continuum.
Then, the Cartwright and Littlewood Theorem \cite{CartwrightLittlewoodFixPointThms} implies 
the existence of a fixed point of $F$ in $C$. 

\end{proof}

\begin{rem}
  We note that as a special case, Corollary~\ref{c.uniquness} applies
  to decomposable essential thin circloids. In order to see this,
  recall that a continuum $C$ is called {\em decomposable} if it can
  be written as the union of two non-empty continua $C_1$ and $C_2$.
  If $C$ is a thin circloid, then due to the minimality of circloids
  $C_1$ and $C_2$ have to be non-essential. Hence, connected
  components $\wh C_i$ of $\pi^{-1}(C_i)\ssq \R^2$, $i=1,2$, are
  bounded. If these lifts are chosen such that their intersection is
  non-empty, then $G=\wh C_1\cup \wh C_2$ is a compact generator of
  $C$.

  For this special case, Corollary~\ref{c.uniquness} was already
  proved in \cite{LeCalvezUniqRotNumDecomp} by using Caratheodory's Prime End Theory.
  Examples of (hereditarily) non-decomposable circloids were
  constructed by Bing \cite{bing:1948} and may occur as minimal sets
  of smooth surface diffeomorphisms \cite{handel:1982,herman:1986}.
\end{rem}

As Lemma~\ref{l.main_generator_lemma} works for any combination of two
compactly generated thin annular continua, we can prove
Theorem~\ref{t.herman} in a similar way as the above
Corollary~\ref{c.uniquness}. However, what we need as a technical
prerequisite is the measurable dependence of the size of the
generators of fibres $h^{-1}(\xi)$ under the assumptions of the
theorem. We obtain this in several steps.  We place ourselves in the
situation of Theorem~\ref{t.herman} and assume again without loss of
generality that the action $h^*:\Pi_1(\torus)\to\Pi_1(\kreis)$ on the
fundamental group is the projection to the second coordinate. This
implies that the annular continua $\mathcal{A}_\xi = h^{-1}\{\xi\}$
are all of homotopy type $(1,0)$. We denote by $\hat f$ the lift of
$f$ to $\A$ and by $F$ the lift to $\R^2$. Further, we denote by $\hat
h:\A\to\R$ the lift of $h$ to $\A$ and by $H:\R^2\to\R$ the lift to
$\R^2$.

Let $\Omega_0=\{\xi\in\kreis\mid \cA_\xi \textrm{ is thin}\}$,
$\Omega=\pi^{-1}(\Omega_0)$ and $A_\xi=\hat h^{-1}\{\xi\} \ (\xi\in\R)$. Then
all $A_\xi$ are essential annular continua in $\A$, and $A_\xi$ is thin if and
only if $\xi\in\Omega$.  Further, define $A^+_\xi=\partial \cU^+(A_\xi)$ and
$A^-_\xi =\partial \cU^-(A_\xi)$.  Then for all $\xi\in\Omega$ we have
$A_\xi=A^+_\xi\cup A^-_\xi$ and, by Lemma~(\ref{e.core-circloid}), $A^+_\xi\cap
A^-_\xi = C_{A_\xi} =: C_\xi$. 

 We recall that for a metric space $(X,d)$ and
$C,D\subset X$, the Hausdorff distance is defined as
\begin{equation*}
d_\mathcal{H}(C,D) = \max\{\sup_{x\in C} d(x,D),\sup_{y\in D}d(y,C)\}.
\end{equation*}
The convergence of a sequence $\{C_n\}_{n\in \N}$ of subsets in $X$ to $A\subset
X$ in this distance is denoted either by $C_n\rightarrow_{\mathcal{H}}A$ or by
$\lim^\mathcal{H}_{n\to\infty} C_n = A$. Note that $d_\mathcal{H}(C,D)<\eps$ if
and only if $C\ssq B_\eps(D)$ and $D\ssq B_\eps(C)$, and that the Hausdorff
distance defines a complete metric if one restricts to compact subsets.

\begin{lemma} \label{l.ac_convergence} If $A_\xi$ is thin, then
  $\lim_{\xi'\nearrow \xi}^{\mathcal{H}}A^-_{\xi'}=\lim_{\xi'\nearrow \xi}^{\mathcal{H}}A_{\xi'} =
  A^-_\xi$ and $\lim_{\xi'\searrow \xi}^{\mathcal{H}}A^+_{\xi'}=\lim_{\xi'\searrow \xi}^{\mathcal{H}}A_{\xi'} = A^+_\xi$.
\end{lemma}
\begin{proof}
  We prove $\lim_{\xi'\nearrow \xi}^{\mathcal{H}}A^-_{\xi'}=\lim_{\xi'\nearrow
    \xi}^{\mathcal{H}}A_{\xi'} = A^-_\xi$, the opposite case follows by symmetry.
  Since $A^-_{\xi_n}\ssq A_{\xi_n}$, it suffices to show that for all
  $\eps>0$ there exists $\delta>0$ such for all
  $\xi'\in(\xi-\delta,\xi)$ we have
  \begin{equation}\label{e.inclusions}
  A_{\xi'}\ssq B_\eps(A^-_\xi) \quad \textrm{and} \quad \ A_\xi^-\ssq
  B_\eps(A_{\xi'}^-) \ .
  \end{equation}
  We start by showing the first inclusion. Fix $\eps>0$.  Assume for a
  contradiction that there exists a sequence $\xi_n\nearrow \xi$ such that
  $A_{\xi_n}\subsetneq B_\eps(A^-_\xi)$ for all $n\in\N$. Let $z_n\in
  A_{\xi_n}\smin B_\eps(A^-_\xi)$ and $z=\nLim z_n$. Then $z\notin
  B_\eps(A^-_\xi)$ and thus, since all the $z_n$ are below $A_\xi$, we have
  $z\notin A_\xi$.  However, at the same time $h(z)=\nLim h(z_n) = \nLim \xi_n
  =\xi$, a contradiction.

  Conversely, in order to show the second inclusion in (\ref{e.inclusions}),
  assume for a contradiction that there exists a sequence $\xi_n\nearrow \xi$
  such that $A_{\xi}^-\subsetneq B_\eps(A^-_{\xi_n})$ for all $n\in\N$. Let
  $K_n=A_\xi^-\smin B_\eps(A_{\xi_n}^-)$. Then $\nfolge{K_n}$ is a decreasing
  sequence of non-empty compact sets, such that $K=\ncap K_n\neq
  \emptyset$. Note here that $A^-_\xi\cap B_\eps(A_{\xi_n}^-)=A^-_\xi\cap
  B_\eps(\cU^-(A_{\xi_n}^-))$, which is increasing in $n$.  Let $z\in K$. Then
  $B_\eps(z)\cap A_{\xi_n}^-=\emptyset$ and thus
  $B_\eps(z)\ssq\cU^+(A_{\xi_n}^-)$ for all $n\in\N$.  This implies $h(z')\geq
  \xi$ for all $z'\in B_\eps(z)$, contradicting the fact that $B_\eps(z)$
  intersects $\cU^-(A_\xi)$ and $h < \xi$ on $\cU^-(A_\xi)$.
\end{proof}

Given a compactly generated thin annular continuum $A$, we let 
\[
\tau(A) \ = \ \inf\{\diam(G)\mid G \textrm{ is a compact generator of
} A\} \ .
\]
\begin{lemma} \label{l.generator_measurability} $\xi\mapsto
  \tau(A^-_\xi)$ is lower semi-continuous from the left on $\Omega$,
  that is,
\[
\liminf_{\xi'\nearrow \xi} \tau(A^-_{\xi'}) \ \geq \ \tau(A^-_\xi)
\quad \textrm{for all }\xi\in\Omega \ .
\]
Similarly, $\xi\mapsto \tau(A^+_\xi)$ is lower semi-continuous from
the right on $\Omega$.
\end{lemma}
\begin{proof}
  Let $\xi_n\nearrow \xi$ and assume without lose of generality that
  $\tau:=\nLim \tau(A^-_{\xi_n})$ exists and is finite. Choose
  generators $G_{\xi_n}$ of $A_{\xi_n}$ of diameter smaller than
  $\tau(A^-_{\xi_n})+\ntel$. Then, using Lemma~\ref{l.ac_convergence},
  it is straightforward to verify that any limit point $G$ of
  \nfolge{G_{\xi_n}} in the Hausdorff metric is a compact generator of
  $A^-_\xi$ of diameter smaller than $\tau$.
\end{proof}
It is easy to check that real-valued functions which are lower semi-continuous
from one side are also measurable. Consequently, since $\tau(A_\xi) \leq
\eta(\xi):=\tau(A^-_\xi)+\tau(A^+_\xi)$, the function $\eta$ provides a
measurable majorant for the minimal diameter of the generators of
$A_\xi$. Further, $A^\pm_\xi$ are compactly generated if and only if $A_\xi$ is
compactly generated, a fact which follows from the topological considerations on
thin annular continua exposed in the next section, see
Lemma~\ref{curvarpx} (iv). Altogether, this yields
\begin{cor} \label{c.generator_measurability}
  Assume that $A_\xi$ has compact generator for almost every $\xi\in\T^1$. 
  Then the map $\xi\mapsto \tau(A_\xi)$ has a measurable finite-valued
  majorant.
\end{cor}

We are now in the position to complete the
\begin{proof}[\bf Proof of Theorem~\ref{t.herman}.] Let
  $f\in\homeo_0(\T^2)$ and suppose $h:\torus\to\kreis$ is a
  semiconjugacy to the irrational rotation $R_\rho$. We assume again
  that $h^*=\pi_2^*$, such that there exist a continuous lift
  $\hat h :\A\to\R$ of $h$ and a lift $\hat f:\A\to\A$ of $f$ which satisfy
  \[
            \hat h\circ\hat f \ = \ R_\rho \circ \hat h \ .
  \]
  Let $F:\R^2\to\R^2$ be a lift of $f$. Assume for a contradiction
  that $f$ has no unique rotation number.  As in the proof of
  Corollary~\ref{c.uniquness}, this implies the existence of two
  $f$-invariant ergodic probability measures $\mu_1$ and $\mu_2$ with
  \[
  \rho_1\ = \ \int_{\torus} \pi_2(F(z)-z) \ d\mu_1(z) \ \neq \
  \int_{\torus} \pi_2(F(z)-z) \ d\mu_2(z) \ = \rho_2 \ .
  \]
  As $h^{-1}\{\xi\}$ is compactly generated for Lebesgue-a.e.\
  $\xi\in\kreis$, Corollary~\ref{c.generator_measurability} yields the
  existence of a finite-valued measurable majorant of
  $\xi\mapsto\tau(h^{-1}\{\xi\})$. Hence, we can find a constant
  $C>0$ and a set $\Omega_C\ssq\kreis$ of positive measure such that for
  all $\xi\in\Omega_C$ the annular continuum $h^{-1}\{\xi\}$ has a
  compact generator $G_\xi$ with $\diam(G_\xi)\leq C$.

  Both $\mu_1$ and $\mu_2$ must be mapped to the Lebesgue measure on
  $\kreis$ by $h$, since this is the only invariant probability
  measure of $R_\rho$. Hence, for almost every $\xi\in\kreis$ there
  exist points $z_1,z_2\in h^{-1}\{\xi\}$ which are generic with
  respect to $\mu_1$ and $\mu_2$, respectively. In particular, for any
  lift $\hat z_i\in\R^2$ of $z_i$ we have that
  \begin{equation}\label{e.rho_i}
    \nLim \pi_2(F^n(\hat z_i)-\hat z_i)/n \ = \ \rho_i \quad (i=1,2) \ .
  \end{equation}
  Without loss of generality, we may assume that $h^{-1}\{\xi\}$ has
  compact generator $G_\xi$ and $R_\rho^n(\xi)$ visits $\Omega_C$
  infinitely many times, say, $r^{n_i}_\rho(\xi)\in\Omega_C$ for a strictly
  increasing sequence $\ifolge{n_i}$ of integers. Given lifts $\hat
  z_1,\hat z_2\in G_\xi$ of $z_1,z_2$,
  Lemma~\ref{l.main_generator_lemma} implies that 
  \[
  \pi_2(F^{n_i}(\hat z_1))-\pi_2(F^{n_i}(\hat z_2) \ \leq \
  \diam(G_{r^{n_i}_\rho(\xi)}) + \nu_{G_\xi}+\nu_{G_{r^{n_i}_\rho(\xi)}}+1 \ \leq \ \nu_{G_\xi}+2C+1 \
  \]
  for all $i\in\N$. As $\rho_1\neq \rho_2$, this contradicts
  (\ref{e.rho_i}).  
\end{proof}

\section{Comments on the Topology and Rotation Sets of thin annular
  continua.}\label{TopClassTACandCommRotTh}

Given $X\subset \R^2$ we denote by $[X]_y$ the connected component of $y$ in $X$
and define $h(X)=\sup\{x_1-x_2\mid(x_1,y_1),(x_2,y_2)\in X\}$. For an essential
annular continuum $A\subset \A$, we define the set of {\em spikes} of $A$ as
$$\mathcal{S}_{A}:=\{[(\hat{A}\setminus C_{\hat{A}})]_{x}\mid x\in\hat{A}\setminus C_{\hat{A}}\} \ $$
and say that $A$ has an infinite spike if there exists $S\in\mathcal{S}_A$ with
$h(S)=\infty$.  Further we let $H_{\mathcal{S}_{A}}:=\sup\{h(S)\mid
S\in\mathcal{S}_{A}\}$. The main result of this section is the following.
\begin{thm}\label{Characterisation}
Let $A\subset\A$ be an essential thin annular continuum. Then the following holds.
\begin{itemize}
\item[(1)] If $A$ is not compactly generated then either
\begin{itemize}
\item[(1a)] $C_A$ is not compactly generated, or
\item[(1b)] $C_A$ is compactly generated and $A$ contains an infinite spike.
\end{itemize}
\item[(2)] If $A$ compactly generated, then so is $C_A$.
\end{itemize}
\end{thm}
The proof is given in Section~\ref{proofTopcharTAC} below.  As a example in the
class of continua given in (1a), we can regard the {\em Birkhoff
  attractor}. This is an essential thin circloid which has a segment as a
rotation set for some map that leaves it invariant (see
e.g. \cite{LeCalvez1988BirkhoffAttractor}). Hence, due to Corollary
\ref{c.uniquness} the Birkhoff attractor cannot have a compact generator. For
the class given in (1b) we can consider the continuum given by
$A=\pi\left(\R\times\{0\}\cup \left\{\left(x,\frac{1}{x}\right)\in\R^2\mid x\geq
    1\right\}\right)$, which contains the infinite spike
$S=\left\{\left(x,\frac{1}{x}\right)\in\R^2\mid x\geq 1\right\}$.

Note that Theorem~\ref{Characterisation} does not rule out the coexistence of an
infinite spike and a compact generator. In fact, this may happen, and a way to
construct such examples is the following. Let $I=[0,1]\times\{0\}$ and
$J=\{0\}\times[0,1]$. We consider $K=J\cup I\cup T(J)$.  Fix $x_0\in J\smin I$
and $x_1=T(x_0)$ and and let $\gamma:[0,+\infty)\rightarrow
\{(x,y)\in\R^2|0<x<1,y>0\}$ be an injective curve that verifies
\begin{itemize}
\item[(i)] $\gamma([n,+\infty))\subset B_{\frac{1}{n}}(K)$ for every $n\in \N$;
\item[(ii)] $\lim_i \gamma(t_i)=x_0$, $\lim_j \gamma(t_j)=x_1$ for two strictly increasing 
sequences of positive integers $(t_i)_{i\in\N},(t_j)_{j\in\N}$.
\end{itemize}
Now let $A=\pi(\tilde{A})$ where $\tilde{A}:=\bigcup_{n\in\Z}T^n(K\cup\gamma)$.
It is easy to see that $A$ is a thin essential annular continuum. Furthermore
the set $G=K\cup\gamma$ is compact and connected, and hence a generator of
$A$. Finally the set $S:=\tilde{A}\setminus \left(\R\times\{0\}\right)$ is
connected since $S=\bigcup_{n\in\Z}T^n(\left(J\cup T(J)\setminus I\right)\cup
\gamma)$.  Hence $A$ has compact generator $G$ and at the same time contains the
infinite spike $S$. What is not clear to us is whether similar examples can be
produced with an infinite spike that is not $T$-invariant.
\begin{question}
  Suppose $A$ is a thin annular continuum which contains an infinite spike
  $S$ with $T^n(S)\cap S=\emptyset$ for all $n\in\N$. Does this imply that $A$
  has no compact generator?
\end{question}

As seen in Corollary \ref{c.uniquness}, an invariant compactly generated thin
annular continuum has a unique rotation vector. As the Birkhoff attractor shows,
this is not true if $C_A$ has no generator. For the case of thin annular
continua which are not compactly generated, but have a compactly generated
circloid, the situation is similar.
\begin{prop}\label{RotSetTACwithInfSpikeCAgen}
  Given any segment $I\subset \R$, there exists a map $f\in\homeo(A)$ which
  leaves invariant an essential thin annular continuum $A\subset \A$ such that
  $C_A$ has compact generator, $A$ has an infinite spike, and $\rho_A(f)=I$.
\end{prop}
The proof, which is similar to a construction by Walker \cite{WalkerPrimeEnd} of
some examples in the context of prime end rotation, is given in
Section~\ref{rotsetcompcircinfspike}.

\begin{subsection}{Topology of thin annular continua: Proof of Theorem \ref{Characterisation}.}\label{proofTopcharTAC}

  The proof of Theorem~\ref{Characterisation} mainly hinges on a number of
  purely plane-topological lemmas. In order to state and prove them, we need to
  introduce further notation.  Given a Jordan curve $\gamma\subset \R^2$ we
  denote its Jordan domain by $D_{\gamma}$. Further, for two curves
  $\alpha,\beta$ in the plane with $\alpha(0)=\beta(1)$, we denote their
  concatenation by $\alpha\cdot\beta$. We say an arc is an injective curve
  $\alpha:[0,1]\to\R^2$ and define
  $\stackrel{\circ}{\alpha}:=\alpha\setminus\{\alpha(0),\alpha(1)\}$. We will
  sometimes abuse terminology by identifying the function defining a curve and
  its image.  The diameter of a set $X\subset \R^2$ is denoted by
  $\textrm{diam}(X)$. As before, given two sets $A,B\subset\R^2$ we denote the
  Hausdorff distance by $d_\cH(A,B)$.  Our first step is to prove that the
  existence of a generator for an essential thin annular continuum $A$ implies
  the existence of a generator for the circloid $C_A$. We divide the proof into
  several lemmas.

\begin{lemma}\label{regularneigh}
  Given a bounded set $X\subset \R^2$ and two positive numbers
  $\delta_1<\delta_2$, there exists a neighbourhood
  $U_{\delta_1\delta_2}(X)$ of $X$ which verifies
\begin{itemize}
\item[(i)] $B_{\delta_1}(X)\subset U_{\delta_1\delta_2}(X)\subset
  B_{\delta_2}(X)$;
\item[(ii)] $U_{\delta_1\delta_2}(X)=\bigcup_{i=1}^n D_i$ where
  $\left\{\overline{D_i}\right\}_{i=1}^n$ is a family of pairwise
  disjoint topological closed disks with $\mathcal{C}^1$-boundaries $\partial D_i$.
\end{itemize}
\end{lemma}
\begin{proof}
  Let $\Delta:\R^2\rightarrow \R$ be given by
  $\Delta(y)=d(y,\overline{X})$. Let
  $\varepsilon_1=\min\left\{\frac{\delta_1}{10},\frac{\delta_2-\delta_1}{10}\right\}$,
  such that $a:=\delta_1+\eps_1<\delta_2+\eps_2=:b$.  Further, let
  $\Psi:\R^2\rightarrow \R$ be a $C^{\infty}$ function such that
  $d_0(\Psi,\Delta)<\varepsilon_1$.  Then $\Psi^{-1}((-\infty,a)) \supset
  \ B_{\delta_1}(\overline{X})$ and $\Psi^{-1}((-\infty,b))\subset
  B_{\delta_2}(\overline{X})$.

  Consider now a regular value \footnote{That is, a value $r$ such that for all
    $y\in \Psi^{-1}(r)$ the derivative matrix $D\Psi(y)$ is surjective.}
  $r\in(a,b)$ of $D$ which exists due to Sard's Theorem \cite[Chapter 1, Section
  7]{GuilPolDiffTop}. Then we have that $\Psi^{-1}(r)$ is a union of pairwise
  disjoint $C^1$ Jordan curves $\{C_j\}_{j=1}^m$. Hence,
  $U_{\delta_1\delta_2}(X):=\Psi^{-1}((-\infty,r))$ has all the desired
  properties. Note here that $X\ssq \Psi^{-1}((-\infty,r))$ and $\partial
  \Psi^{-1}((-\infty,r))\ssq \Psi^{-1}(\{r\})$.
\end{proof}

\begin{lemma}\label{discos}
  Let $B_1,B_2$ be two open sets in $\R^2$ with
  $d(B_1,B_2)=\varepsilon_0>0$ and fix $\varepsilon_1>0$.  Then there
  exists $\xi=\xi(\eps_0,\eps_1)>0$ such that given any pair of arcs
  $\alpha$ and $\beta$ which verify
\begin{itemize}
\item[(i)] $\alpha(0)=\beta(1)\in B_1$, $\alpha(1)=\beta(0)\in B_2$
  and $\alpha\cdot\beta$ is a Jordan curve,
\item[(ii)] $d(\alpha\setminus \left(\overline{B_1\cup
      B_2}\right),\beta\setminus \left(\overline{B_1\cup
      B_2}\right))\geq \varepsilon_1$,
\end{itemize}
there exists $z\in D_{\alpha\cdot\beta}$ for which $B_\xi(z)\subset
D_{\alpha\cdot\beta}\setminus \left(\overline{B_1\cup B_2}\right)$.  (See Figure
\ref{figd}).
\end{lemma}
\begin{figure}[ht]\begin{center}
\psfrag{B1}{$B_1$}\psfrag{B2}{$B_2$}\psfrag{alpha}{$\alpha$}
\psfrag{beta}{$\beta$}\psfrag{alpha0}{$\alpha(0)=\beta(1)$}
\psfrag{alpha1}{$\alpha(1)=\beta(0)$}\psfrag{Bzxi}{$B_{\xi}(z)$}
\includegraphics[height=4cm]{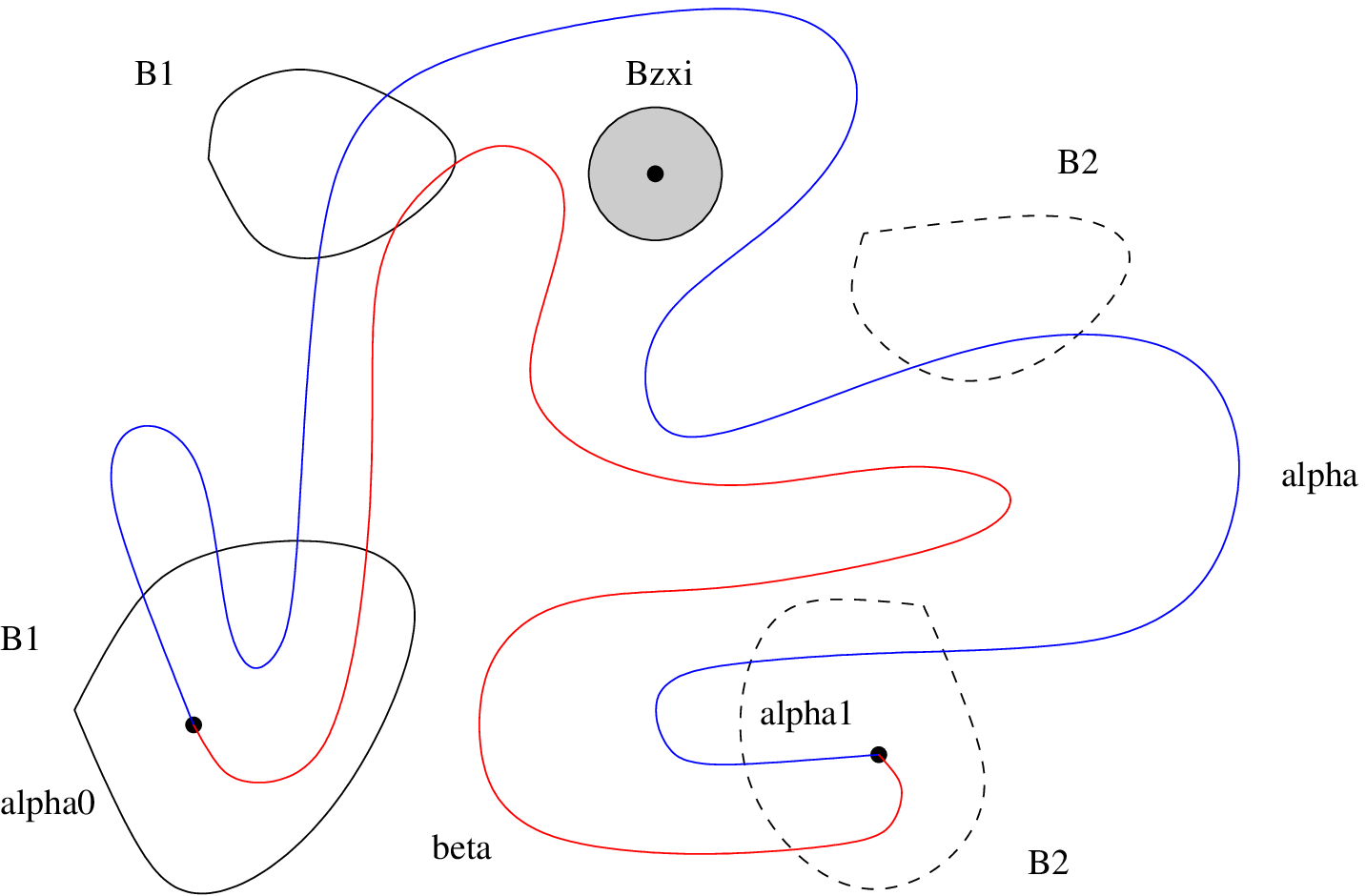}
\caption{}\label{figd}
\end{center}\end{figure}
\begin{proof}
  Let $\delta_1=\frac{\varepsilon_0}{10}$ and
  $\delta_2=\frac{\varepsilon_0}{5}$. We consider
  $U_1:=U_{\delta_1\delta_2}(B_1)$ and
  $U_2:=U_{\delta_1\delta_2}(B_2)$ given by the previous lemma. Due to
  the fact that continuous arcs can be approximated by $C^1$ arcs, we
  may assume that the arcs $\alpha,\beta$ are $C^1$. Furthermore, for
  similar reasons we may assume that the curves $\alpha,\beta$ are
  transverse to each Jordan curve in $\partial U_1 \cup \partial U_2$.
  
  Due to $U_1$ being disjoint of $U_2$, we have a $C^1$ loop
  $\theta\subset\partial U_1$ that separates $\alpha(0)$ from
  $\alpha(1)$. We claim that there exists a subarc $J\ssq \theta$ such
  that 
\begin{itemize}
\item[(a)] $J(0)\in\alpha\cap\theta$ and $J(1)\in \beta\cap \theta$;
\item[(b)] $\stackrel{\circ}{J}\subset D_{\alpha\cdot\beta}$.
\end{itemize}
In order to see this, fix a point $x_0\in\theta$ outside of
$\overline{D_{\alpha\cdot\beta}}$ and assume, by reparametrising if necessary,
that $\theta(0)=x_0$. Then, if $J$ with the above properties does not exist,
this means that every time that $\theta$ goes into $D_{\alpha\cdot\beta}$ by
passing through $\alpha$, it goes out of $D_{\alpha\cdot\beta}$ through $\alpha$
as well. Consequently, the algebraic intersection number $[\alpha]\wedge
[\theta]$ between the homotopy classes $[\alpha],[\theta]$ relative to the
points $x_0,\alpha(0),\alpha(1)$, is zero. However, this is a contradiction
since $\theta$ separates $\alpha(0)$ and $\alpha(1)$, so $[\alpha]\wedge
[\theta]=1$ (compare \cite[Chapter 3]{GuilPolDiffTop}).

Therefore, we can consider an arc $J$ with the above properties.  Let
$\xi:=\min\{\frac{\varepsilon_0}{10},\frac{\varepsilon_1}{2}\}$ and
consider a bijective parametrisation $\gamma:[0,1]\rightarrow J$ with
$\gamma(0)=z_0\in\alpha$ and $\gamma(1)=z_1\in\beta$ (see Figure
\ref{proofdiscos}).

\begin{figure}[ht]\begin{center}
    \psfrag{parB1}{$\partial
      U_1$}\psfrag{J}{$J$}\psfrag{alpha}{$\alpha$}
       \psfrag{alpha0}{$\alpha(0)=\beta(1)$}\psfrag{beta}{$\beta$}
    \psfrag{z0}{$z_0$}\psfrag{z1}{$z_1$}
\includegraphics[height=4cm]{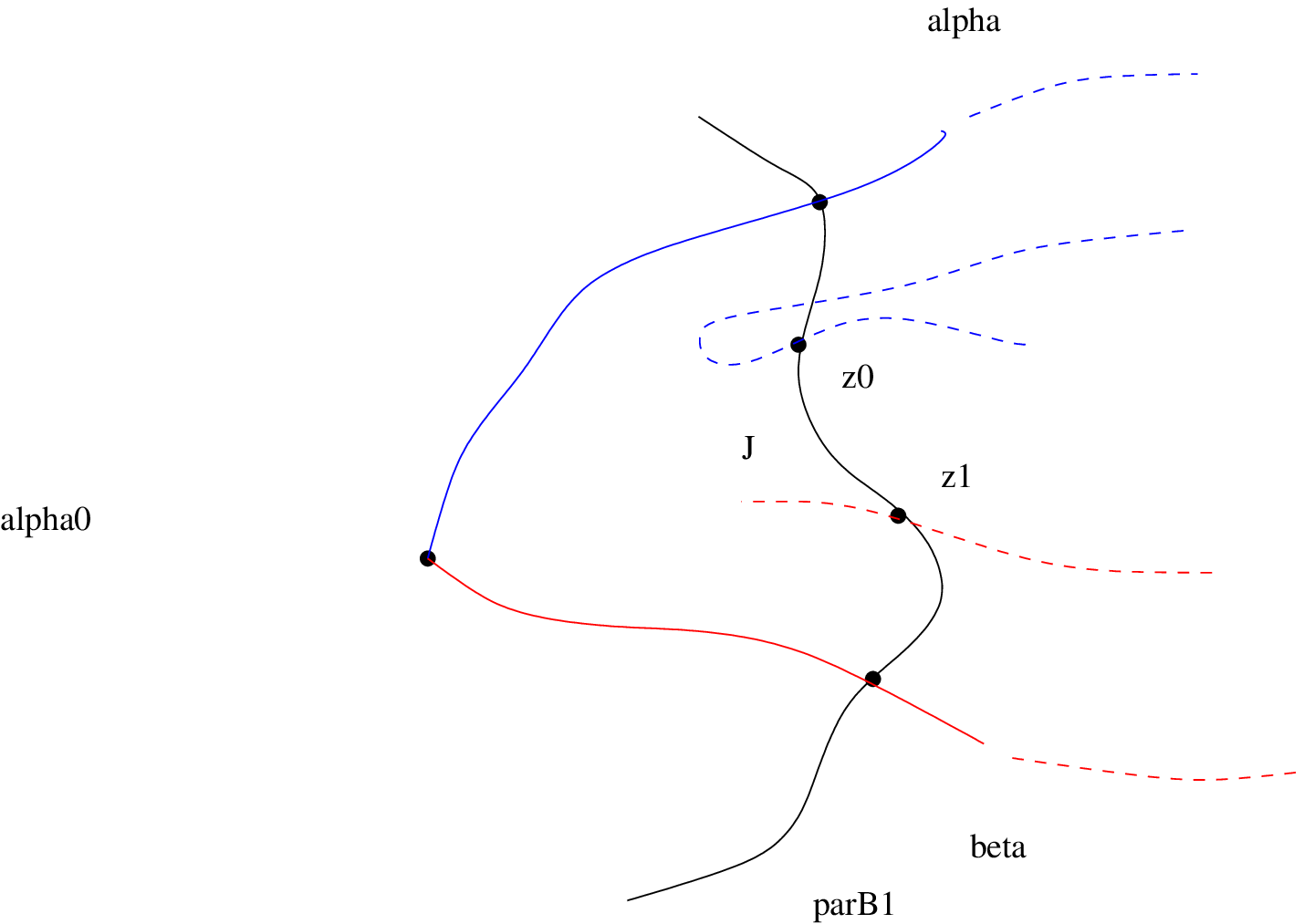}
\caption{}\label{proofdiscos}
\end{center}\end{figure}

Let $\lambda:=\sup\{t\in[0,1]\mid d(\gamma(t),\alpha)<\xi\}$ which
exists since $1$ is an upper bound of $\{t\in[0,1]\mid
d(\gamma(t),\alpha)<\xi\}$.  Otherwise we would have that
$d(\alpha,\beta)<\varepsilon_1$ in $(B_1\cup B_2)^c$.

Then, we have for $z:=\gamma(\lambda)$ that $B_\xi(z)\cap\alpha=\emptyset$ and
$B_{\xi+\delta}(z)\cap \alpha\neq\emptyset$ for any positive number
$\delta$. Thus we have $B_\xi(z)\cap \beta=\emptyset$, otherwise
$d(\alpha,\beta)<\varepsilon_1$ in $\left(\overline{B_1\cup B_2}\right)^c$ which
is impossible.  This implies that $B_\xi(z)\subset D_{\alpha\cdot\beta}\setminus
\left(\overline{B_1\cup B_2}\right)$.
\end{proof}

\begin{lemma}\label{l1l2connected} Suppose $(\alpha_n)_{n\in\N}$ and
  $(\beta_n)_{n\in\N}$ are two sequences of arcs in $\R^2$ with the
  following properties:
\begin{itemize}
\item[(i)] $\alpha_n,\beta_n\subset B_R(0)$ and
  $\alpha_n\cap\beta_n=\emptyset$ for every $n\in\N$ and some
  $R\in\R$;
\item[(ii)]
  $\lim^\mathcal{H}_{n\to\infty}\alpha_n=\mathcal{L}_1,\lim^\mathcal{H}_{n\to\infty}\beta_n=\mathcal{L}_2$;
\item[(iii)] $(\alpha_n(0))_{n\in\N}$ and $(\beta_n(0))_{n\in\N}$
  converge to $x_0\in\R^2$, $(\alpha_n(1))_{n\in\N}$ and
  $(\beta_n(1))_{n\in\N}$ converges to $x_1\in\R^2$ with $x_0\neq
  x_1$.
\end{itemize}
 Then we have that either
$[\mathcal{L}_1\cap\mathcal{L}_2]_{x_0}=[\mathcal{L}_1\cap\mathcal{L}_2]_{x_1}$
or $\partial(\mathcal{L}_1\cup\mathcal{L}_2)$ separates $\R^2$.
% (See Figure \ref{figgencirc}).
\end{lemma}
% \begin{figure}[ht]\begin{center}
% \psfrag{alphan}{$\alpha_n$}\psfrag{betan}{$\beta_n$}
% \psfrag{x0}{$x_0$}\psfrag{x1}{$x_1$}
% \includegraphics[height=3cm]{gencirc.eps}
% \caption{}\label{figgencirc}
% \end{center}\end{figure}
\begin{proof}
  Suppose $[\mathcal{L}_1\cap\mathcal{L}_2]_{x_0}\neq
  [\mathcal{L}_1\cap\mathcal{L}_2]_{x_1}$. We claim that for some $\delta>0$
  there exist two different connected components $B_1,B_2$ of
  $B_\delta(\cL_1\cup\cL_2)$, such that
\begin{itemize}
\item $d(\overline{B}_1,\overline{B}_2)=\varepsilon_0>0$;
\item if $C_1=B_1\cap\left(\mathcal{L}_1\cap\mathcal{L}_2\right)$ and
  $C_2=B_2\cap\left(\mathcal{L}_1\cap\mathcal{L}_2\right)$ then we
  have $x_0\in C_1,x_1\in C_2$ and
  $\mathcal{L}_1\cap\mathcal{L}_2=C_1\cup C_2$.
\end{itemize}
Otherwise for every $\varepsilon>0$ the points $x_0$ and $x_1$ would
be in the same connected component of
$B_\eps(\mathcal{L}_1\cap\mathcal{L}_2)$. However, this would
imply that
$K:=\bigcap_{n\in\N}[B_{\ntel}(\mathcal{L}_1\cap\mathcal{L}_2)]_{x_0}
\subset\mathcal{L}_1\cap\mathcal{L}_2$ is a connected set containing
$x_0$ and $x_1$, contradicting $[\cL_1\cap\cL_2]_{x_0}\neq [\cL_1\cap \cL_2]_{x_1}$. 

Hence, we can choose $B_1,B_2$ as above. Then, for every $n\in\N$ we consider
small arcs $a_n,b_n,c_n$ and $d_n$ such that $a_n$ joins $\alpha_n(1)$ with
$x_1$, $b_n$ joins $x_1$ with $\beta_n(1)$, $c_n$ joins $\beta_n(0)$ with $x_0$
and $d_n$ joins $x_0$ with $\alpha_n(0)$. Further we may request that
$\gamma_n:=\alpha_n\cdot a_n\cdot b_n\cdot \beta_n^{-1}\cdot c_n\cdot d_n$ is a
$C^1$ Jordan curve, and that $\lim^\mathcal{H}_{n\to\infty}a_n\cdot b_n=x_1$,
$\lim^\mathcal{H}_{n\to\infty}c_n\cdot d_n=x_0$ (See Figure \ref{proofgencirc}).

\begin{figure}[ht]\begin{center}
    \psfrag{B1}{$B_1$}\psfrag{B2}{$B_2$}\psfrag{alphan}{$\alpha_n$}
    \psfrag{betan}{$\beta_n$}\psfrag{an}{$a_n$}
    \psfrag{bn}{$b_n$}\psfrag{cn}{$c_n$}\psfrag{dn}{$d_n$}
    \psfrag{x0}{$x_0$}\psfrag{x1}{$x_1$}
\includegraphics[height=3cm]{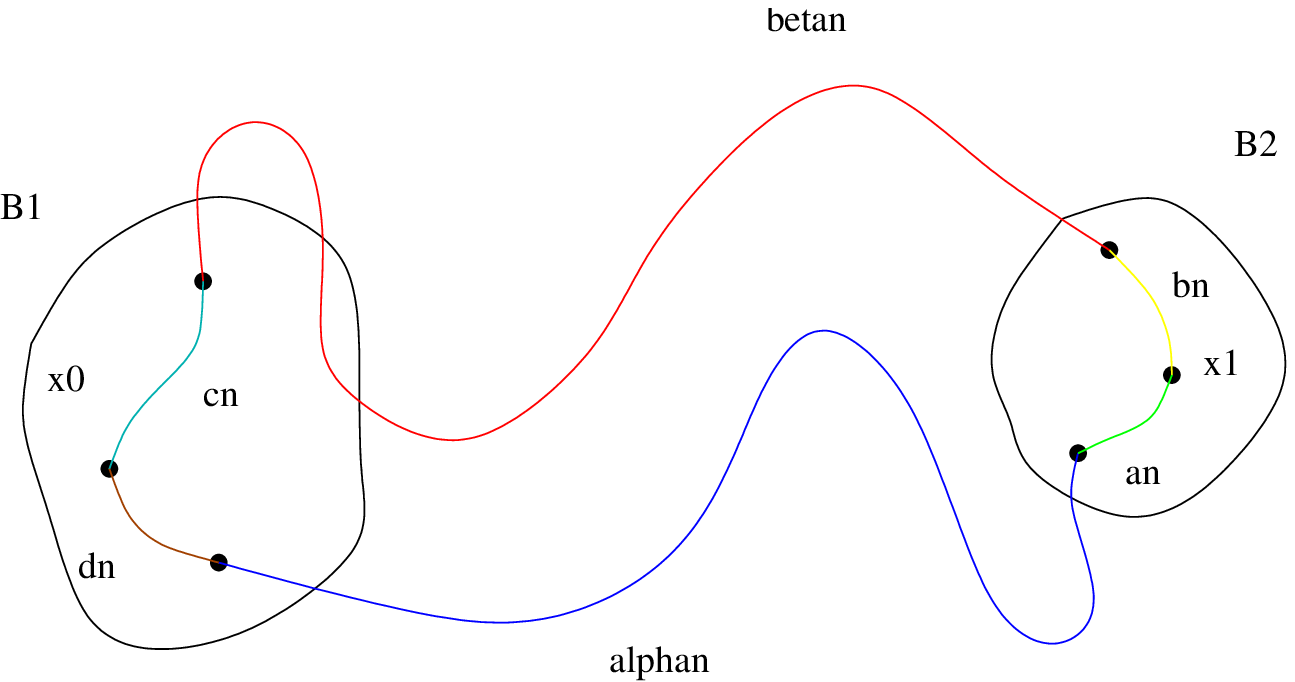}
\caption{}\label{proofgencirc}
\end{center}\end{figure}

Let $\omega_n:=d_n\cdot \alpha_n\cdot a_n$ and $\theta_n:=b_n\cdot
\beta_n\cdot c_n$, $D_n:=D_{\omega_n\cdot\theta_n}=D_{\gamma_n}$.
Since $\mathcal{L}_1\cap\mathcal{L}_2$ is contained in $B_1\cup B_2$
and we have $\lim^\cH_{n\to\infty} \omega_n=\cL_1$ and $\nLim^\cH
\theta_n=\cL_2$, there exists $\varepsilon_1>0$ such that for every
$n\in\N$
 $$ d(\omega_n\setminus\left(\overline{B_1\cup B_2}\right),
 \theta_n\setminus\left(\overline{B_1\cup B_2}\right))\geq
 \varepsilon_1$$

 Consequently, due to Lemma \ref{discos} there exists a positive
 number $\xi$ and a sequence of points $\left(z_n\right)_{n\in\N}$
 such that $B_\xi(z_n)\subset D_n\setminus\left(\overline{B_1\cup
     B_2}\right)$ for every $n\in\N$. Therefore, if we consider a
 subsequence $(D_{n_i})_{i\in\N}$ which converges to $D_0$, we have
 the existence of some $z_0\in D_0$ such that $B_\xi(z_0)\subset
 D_0\setminus \left(\overline{B_1\cup B_2}\right)$. By construction,
 every curve joining $z\in B_\xi(z_0)$ with $\infty$ intersects $\partial D_0\subset
 \partial \left(\mathcal{L}_1\cup\mathcal{L}_2\right)$. Hence
 $\partial \left(\mathcal{L}_1\cup\mathcal{L}_2\right)$ separates
 $\R^2$ into at least two components.
\end{proof}

\begin{lemma}\label{discaprx}
  Let $D\subset \R^2$ be a smooth disk such that $ \inte(D)\cap
  \inte(T(D))\neq\emptyset$.  Further consider the embedding of the real line
  $\Gamma_1=\partial\mathcal{U}^-\left(\bigcup_{n\in\Z}T^n(\overline{D})\right)$.
  Then for every $x\in\Gamma_1$ and every arc $\gamma_x$ with endpoints
  $x,T(x)$, we have $h(\gamma_x)<4\cdot\textrm{diam}(D)+4$.  The analogous
  statement holds for
  $\Gamma_2=\partial\mathcal{U}^+\left(\bigcup_{n\in\Z}T^n(\overline{D})\right)$.
\end{lemma}

\begin{proof}

  We first fix $x\in\Gamma_1$ with the property that $\pi_2(x)=\max
  \pi_2(\Gamma_1)$. Let $D_k:=T^k(D)$. Suppose for a contradiction that
  $h(\gamma_x)\geq 2\cdot\textrm{diam}(D)+2$ .  Let $z\in\gamma_x$ maximise the
  function $|\pi_1(.)-\pi_1(x)|:\gamma_x\rightarrow \R$.  Then
  $|\pi_1(z)-\pi_1(x)|\geq\textrm{diam}(D)+1$. Further, due to the
  $T$-invariance of $\Gamma_1$ and the choice of $x$ we have that $\gamma_x$
  verifies $\gamma_x\cap \left(\R\times\{\pi_2(x)\}\right)=\gamma_x\cap s_x$,
  where $s_x$ is the interval joining $x$ with $T(x)$.

  Let $\Delta_x\subset\gamma_x$ be an arc containing $z$, such that
  $\Delta_x(0)=x_0,\Delta_x(1)=x_1\in s_x$ and $\Delta_x(t)\notin s_x$ for all
  $t\in (0,1)$. Further, let $I_x$ be the segment joining $x_1$ with $x_0$ and
  $\alpha_x:=\Delta_x\cdot I_x$. Then $D_{\alpha_x}\subset
  \mathcal{U}^+(\Gamma_1)$ since $\stackrel{\circ}{I_x}\subset
  \mathcal{U}^+(\Gamma_1)$.  This implies for $D_k$ with $z\in\partial D_k$ that
  $D_k\subset D_{\alpha_x}$. Otherwise we would have $D_k\cap
  \stackrel{\circ}{I_x}\neq\emptyset$ which implies that
  $\textrm{diam}(D_k)>\textrm{diam}(D)$. On the other hand, we have by
  construction that $T(D_{\alpha_x})\cap D_{\alpha_x}=\emptyset$.  Hence
  $T(D_k)\cap D_k=\emptyset$, which is absurd.

  Given now any point $y\in\Gamma_1$ we find $v=(n,0)$ with $n\in\Z$ such that
  $y\in\gamma_{x+v}$. Thus $T(y)\in\gamma_{x+v+(1,0)}$, and $\gamma_y$ is an arc
  contained in $\gamma_{x+v}\cup\gamma_{x+v+(1,0)}$. This implies that
  $h(\gamma_y)<2\cdot h(\gamma_x)<4\cdot\textrm{diam}(D)+4$.
\end{proof}

\begin{lemma}\label{curvarpx}
  Let $A\subset\A$ be an essential thin annular continuum with compact
  generator. Then, there exist two sequences of arcs
  $(\alpha_n)_{n\in\N}$ and $(\beta_n)_{n\in\N}$ which verify:
\begin{itemize}
\item[(i)] $\alpha_n\subset\mathcal{U}^-(\hat{A})$ and
  $\beta_n\subset\mathcal{U}^+(\hat{A})$ for every $n\in\N$;
\item[(ii)] $(\alpha_n(0))_{n\in\N}$, $(\beta_n(0))_{n\in\N}$
  converges to $x_0\in C_{\hat{A}}$ and $(\alpha_n(1))_{n\in\N}$,
  $(\beta_n(1))_{n\in\N}$ converges to $x_1=T(x_0)$;
\item[(iii)] $\lim^\mathcal{H}_{n\to\infty}\alpha_n=\mathcal{L}_1$ and
  $\lim^\mathcal{H}_{n\to\infty}\beta_n=\mathcal{L}_2$ exist;
\item[(iv)] $\mathcal{L}_1$ and $\mathcal{L}_2$ are generators of $A^-=\partial
  \mathcal{U}^-(A)$ and $A^+=\partial \mathcal{U}^+(A)$, respectively, and
  $G:=\mathcal{L}_1\cup\mathcal{L}_2$ is a generator of $A$.
\end{itemize}
 \end{lemma}
\begin{proof}
  Let $G'$ be a generator of $A$. Due to the Riemann Mapping Theorem applied in
  $\R^2\cup\{\infty\}\setminus G'$, we can consider a sequence of smooth Jordan
  curves $(\gamma_n)_{n\in\N}$ such that
\begin{itemize}
\item[(1)] $\gamma_n\subset \R^2\setminus G'$ for every $n\in\N$;
\item[(2)] $\lim^\mathcal{H}_{n\to\infty}\gamma_n=G'$;
\item[(3)] $D_{n+1}\subset D_n$, where $D_n=D_{\gamma_n}$ for every $n\in\N$. 
\end{itemize}
Due to (2) we have $\textrm{diam}(D_n)<K_0$ for some $K_0\in\R$ and
all $n\in\N$.  For every $n\in\N$, $m\in\Z$ we define
$D_{n,m}=T^m(D_n)$, $\mathcal{D}_n=\bigcup_{m\in\Z}D_{n,m}$, and
$\Gamma^-_n=\partial \mathcal{U^-}\left(
  \mathcal{D}_n\right),\Gamma^+_n=\partial \mathcal{U^+}\left(
  \mathcal{D}_n\right)$. Notice that the curves
$\Gamma^-_n,\Gamma^+_n$ are embeddings of the real line of the same
form as in the previous lemma.  Let $x_0\in G'\cap C_{\hat{A}}$.

Given $n\in\N$ we consider two arcs $\alpha'_{n}\subset \Gamma^-_n$
and $\beta'_{n}\subset \Gamma^+_n$ such that
$\alpha'_n(1)=T(\alpha'_n(0))$, $\beta'_n(1)=T(\beta'_n(0))$ and
$(\alpha'_n(0))_{n\in\N},(\beta'_n(0))_{n\in\N}$ converge to $x_0$. Due to Lemma
\ref{discaprx} we know that $\alpha'_n,\beta'_n$ are contained in a
ball $B_R(x_0)$ where
$$R=8(K_0+1)+\max\{d(\alpha'_n(0),\beta'_n(0))
\mid n\in\N\} \ .$$ Then by
choosing subsequences of $(\alpha'_n)_{n\in\N},(\beta'_n)_{n\in\N}$ we obtain
two sequences $(\alpha_n)_{n\in\N},(\beta_n)_{n\in\N}$ that verify statements
(i) and (ii) of the theorem and the limits $\cL_1$ and $\cL_2$ exist.  This
implies in particular that $G=\mathcal{L}_1\cup\mathcal{L}_2$ is a subcontinuum
of $\hat{A}$.  Moreover, as $\alpha_n,\beta_n$ are generators of
$\Gamma^-_{k_n},\Gamma^+_{k_n}$ for a suitable sequence $(k_n)_{n\in\N}$ and
$\left(\lim^\mathcal{H}_{n\to\infty}\Gamma^-_{k_n}\cup
  \lim^\mathcal{H}_{n\to\infty}\Gamma^+_{k_n}\right)=\hat{A}$ since $A$ is thin,
we have that $\mathcal{L}_1$ is a generator of $\partial\mathcal{U}^-(A)$,
$\mathcal{L}_2$ is a generator of $\partial\mathcal{U}^+(A)$ and $G$ is a
generator of $A$.
\end{proof}

As consequence of the above results, we obtain the following.
\begin{cor}\label{generator}
  If an essential thin annular continuum $A$ has generator, then the
  circloid $C_A$ has generator.
\end{cor}
\begin{proof}
   Let $G=\partial\left(\mathcal{L}_1\cup \mathcal{L}_2\right)$ be the
  generator of $A$ given by Lemma \ref{curvarpx}. Due to Lemma
  \ref{l1l2connected} we have
  $G':=[\mathcal{L}_1\cap\mathcal{L}_2]_{x_0}=[\mathcal{L}_1\cap\mathcal{L}_2]_{T(x_0)}$.
  Further, since $\cL_1\ssq\partial\cU^+(A)$ and
  $\cL_2\ssq\partial\cU^-(A)$ we have that $\cL_1\cap\cL_2\ssq C_A$ by
  Lemma~\ref{l.core_circloid}. This implies that
  $C:=\pi\left(\bigcup_{n\in\N}T^n(G')\right)$ is an essential annular
  continuum contained $C_A$, and hence $C=C_A$ since $C_A$ is a
  circloid.  Thus $G'$ is a generator for $C_A$.
\end{proof}

We now consider the family of all the essential thin annular continua
$\mathcal{A}$, and let $\mathcal{A}_1$ be the family of those
$A\in\mathcal{A}$ such that $C_A$ has no generator and
$\mathcal{A}_2=\mathcal{A}\smin\mathcal{A}_1$. Then due to the last corollary
$\mathcal{A}_2$ contains those $A\in\mathcal{A}$ which admit compact
generator. In order to finish the proof of Theorem
\ref{Characterisation}, it remains to show that whenever
$A\in\mathcal{A}_2$ has no generator then $A$ contains an infinite
spike. We proceed in two steps and start by showing that every finite
spike has at least one `base point' in $C_A$.
\begin{lemma}\label{shortspike}
  If $A$ is a thin annular continuum and $S$ is a spike of $A$ with
  $h(S)<\infty$, then $\overline{S}\cap C_{\hat{A}}\neq \emptyset$.
\end{lemma}
\begin{proof}
  Fix $x\in S$. For each $n\in\N$ we define $\mathcal{Y}_n$ as the family of all
  curves $\alpha:[0,1]\rightarrow \R^2$ which verify:
\begin{itemize}
\item[(i)] $\alpha(0)\in B_\ntel(x)$;
\item[(ii)] $\alpha(1)\in B_\ntel(C_{\hat{A}})$;
\item[(iii)] $\alpha\subset B_\ntel(\hat{A})$. 
\end{itemize}
Let $\R^2\cup\{\infty\}$ be the compactification of $\R^2$ by the
sphere. Then, we can consider a sequence of curves
$(\alpha_n)_{n\in\N}$ with $\alpha_n\in\mathcal{Y}_{k_n}$ such that:

\begin{itemize}
\item $k_n\nearrow \infty$;
\item $\lim^\mathcal{H}_{n\to\infty}\alpha_n=\mathcal{L}$ in $\R^2\cup\{\infty\}$.
\end{itemize}
Let $\mathcal{L}_x:=\left[\mathcal{L}\setminus
  \left(C_{\hat{A}}\cup\{\infty\}\right)\right]_x$.  Then
$\mathcal{L}_x\subset S$, and hence
$\infty\notin\overline{\mathcal{L}_x}$ since $h(S)<\infty$, so
$\overline{\mathcal{L}_x}$ is compact. By property (ii) of the curves
$\alpha_n$ we have $\cL_x\cap B_\ntel(C_{\hat A})\neq\emptyset$ for
all $n\in\N$, and hence $\overline{\cL_x}\cap C_{\hat A}\neq\emptyset$.
\end{proof}

\begin{cor}\label{gcircplusshort}
 
If $A\in\mathcal{A}_2$ and $H_{\mathcal{S}_{A}}<\infty$, then $A$ has generator.

\end{cor}

\begin{proof}
  Let $G'$ be a generator of $C_A$. For every spike $S$ choose
  $n\in\N$ such that $S':=T^n(S)$ intersects $G'$. Note that this is
  possible due to Lemma~\ref{shortspike}. Since
  $H_{\mathcal{S}_{A}}<\infty$ we have that
  $G:=\overline{\left(G'\cup\bigcup_{S\in\mathcal{S}_A}S'\right)}$ is
  a compact generator of $A$.
\end{proof}

Finally, we show that for every $A\in\mathcal{A}_2$ with
$H_{\mathcal{S}_{A}}=\infty$, there exists an infinite spike contained
in $A$.

\begin{prop}\label{infinityspk}
  Let $A\in\mathcal{A}_2$ with $H_{\mathcal{S}_{A}}=\infty$. Then,
  there exists an infinite spike $S\in\mathcal{S}_A$.
\end{prop}
\begin{proof}
  We assume that the supremum $H_{\mathcal{S}_{A}}$ is obtained by spikes in
  $\mathcal{U}^-(C_{\hat{A}})$, the other case is symmetric.  Suppose for a
  contradiction that $h(S)<\infty$ for every $S\in\mathcal{S}_{A}$.

  Let $x_0\in \hat A \smin \cU^+(C_{\hat A}) =
  \left(\hat{A}\cap\mathcal{U}^-(C_{\hat{A}})\right)\cup C_{\hat{A}}$ such that
$$\pi_2(x_0)=\min \left\{\pi_2(x)\ \left| \ x\in \bigcup_{S\in \mathcal{S}_A} S\cap 
  \mathcal{U}^-(C_{\hat{A}})\right.\right\}.$$ By changing coordinates if necessary, we
may assume that $x_0\notin C_{\hat A}$.

Let $\gamma_{x_0}(t)=x_0+t\cdot(1,0)$ and $S_0\in\mathcal{S}_A$ such
that $x_0\in S_0$.  Then due to Lemma \ref{shortspike} and the fact
that $C_A$ has generator, we can consider a generator $L$ of $C_{A}$
that verifies $L\cap\overline{S_0}\neq\emptyset$ and $L\cap
\overline{T(S_0)}\neq\emptyset$ (see Figure \ref{figinfspk}).

\begin{figure}[ht]\begin{center}
    \psfrag{s0}{$T(S_0)$}\psfrag{s0mas}{$S_0$}\psfrag{L}{$L$}\psfrag{s}{$S$}
    \psfrag{gammas0}{$\gamma_{S_0}$}
\includegraphics[height=3cm]{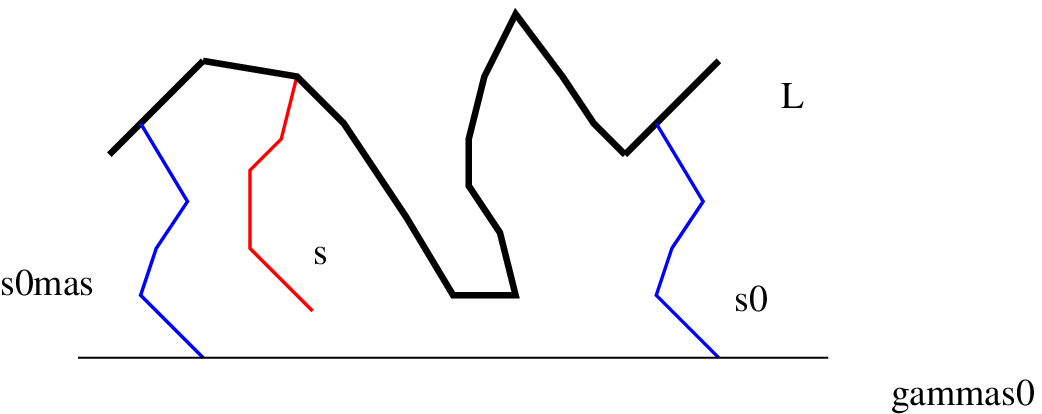}
\caption{}\label{figinfspk}
\end{center}\end{figure}

Given now any spike $S\subset \mathcal{U}^-(C_{\hat{A}})$ different of
$S_0$, due to the definition of $x_0$ we have:

$$ S\subset \left(\overline{\mathcal{U}^+(\gamma_{x_0})} \cap
 \mathcal{U}^-\left(\bigcup_{n\in\Z}T^n(L)\right)\right)\setminus \bigcup_{n\in\Z}T^n(S_0).$$

Therefore we have that $h(S)< 2\cdot h(S_0)+h(L)$. This contradicts
$H_{\mathcal{S}_A}=\infty$.
\end{proof}
This last proposition finishes the proof of Theorem \ref{Characterisation}.
\end{subsection}

\begin{subsection}{Rotation sets for thin annular
    continua.}\label{rotsetcompcircinfspike}

  Let $\mathcal{D}\subset\textrm{Diffeo}_+(\kreis)$ be the set of
  orientation-preserving circle diffeomorphisms with a totally disconnected
  non-wandering set. Note that this means $f\in\cD$ either has rational rotation
  number and a totally disconnected set of periodic points, or $f$ is a Denjoy
  example (with irrational rotation number).

  The aim of this section is to construct a family of examples of
  homeomorphisms $f_{g,\alpha}$ of $\A$, parametrised by
  $g\in\cD$ and $\alpha\in\R$ such that
  \begin{itemize}
  \item $f_{g,\alpha}$ leaves invariant some annular continuum
    $A_{g,\alpha}\in\mathcal{A}_2$ containing at least one infinite spike, and
   \item
    $\rho_{A_{g,\alpha}}(F)=\textrm{conv}(\{\alpha,\rho(g)\})$.\foot{Where
      $\textrm{conv}(X)$ denotes the convex hull of $X$, and $\rho(g)$ is the
      rotation number of $g$.}
  \end{itemize}
  This will prove Proposition~\ref{RotSetTACwithInfSpikeCAgen}.\smallskip

  For any $t\in\R^+$, let $\cR_{t}=\R\times\{t\}$ and define
  $i:(0,+\infty)\to\R$ by $i(x)=\frac{1}{x}$. Further, let
  $\mathcal{G}=\{L_p\}_{p\in\R}$ be the $C^{\infty}$-foliation of
  $\R\times(0,+\infty]$ whose leaves are given by $L_p=\textrm{gr}(i)+(p,0)$ for
  every $p\in\R$, where $\textrm{gr}(i)=\{(x,i(x))|x>0\}$.  Notice that for
  $(x,y)\in\R\times(0,+\infty]$ the leave $l_{(x,y)}$ through $(x,y)$ is given
  by $l_{(x,y)}=L_{p(x,y)}$ with $p(x,y)=x-\frac{1}{y}$.

  Given $g\in\mathcal{D}$, we choose a lift $G:\R\to\R$. Then, we consider
  $F_1:\R\times \R^+\to\R\times \R^+$ given by $F_1(x,y)=(x+v(x,y),y)$ where
  $v(x,y)=G(p(x,y))-p(x,y)$. Notice that $F_1(l_{p(x,y)})=l_{p(x,y)+v(x,y)}$
  for every $(x,y)\in\R\times\R^+$.  Hence, $F_1$ preserves the set
 $$\mathcal{T}:=\bigcup_{p\in\pi^{-1}(\Omega(g))} L_{p}\cap(\R\times [0,1]) \ .$$
 Further, $F_1$ is a $C^{1}$ diffeomorphism since $p$ is $C^{\infty}$ and $G$ is
 $C^1$.

  Let $X:\R\times\R^+\to\R\times\R^+$ be the vector field given by
  $X(x,y)=(\alpha -v(F_1^{-1}(x,y)),t(x,y))$, where $t(x,y)$ is
  uniquely defined in order to have $X(x,y)\in T_{(x,y)}l_{(x,y)}$.
  Then we have that $X$ is $C^1$ (by the same argument that for
  $F_1$), and $\pi_1\circ X(x,y)$ is constant over each leaf of the
  foliation. Let $H$ be the time one of the flow associated to $X$. We
  have that $H$ preserves each leave of the foliation $\mathcal{G}$, and 
  that $\pi_1(H(x,y))=\pi_1(x,y)+\alpha-v(F_1^{-1}(x,y))$ for all
  $(x,y)\in\R^2$. 

  We define $F_2:\R\times\R^+\to\R\times\R^+$ by $F_2(x,y)=H\circ F_1 (x,y)$. Thus
  we have have that $F_2$ is a $C^1$ diffeomorphism which preserves each leave of $\mathcal{G}$, 
  and that verifies $\pi_1(F_2(x,y))=(x,y)+\alpha$ (see Figure \ref{dibujotakinfspikerot}).

\begin{figure}[ht]\begin{center}
    \psfrag{R}{$\cR_0$}\psfrag{R1}{$\cR_1$}\psfrag{Lxy}{$l_{p(x,y)}$}\psfrag{LxyintR1}{$h(x,y)$}
    \psfrag{GLxyintR1}{$G(h(x,y))$}\psfrag{F1xy}{$F_1(x,y)$}\psfrag{F2xy}{$F_2(x,y)$}
    \psfrag{lG}{$l_{p(x,y)+v(x,y)}$}\psfrag{alpha}{$\alpha$}\psfrag{xy}{$(x,y)$}
\includegraphics[height=5cm]{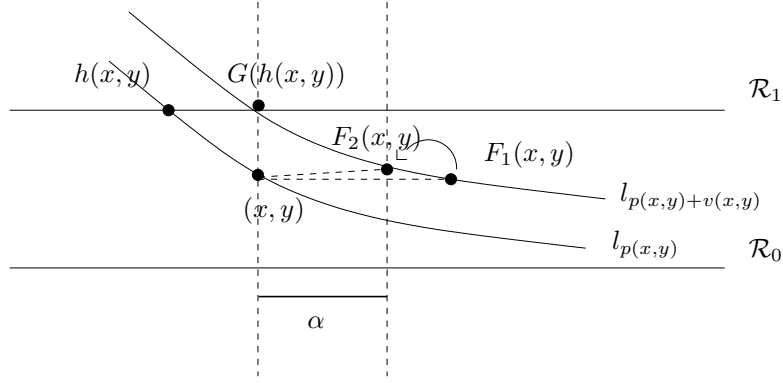}
\caption{As the leaves of $\mathcal{G}$ becomes horizontal as $y$ tends to zero, we have 
that $t(x,y)$ tends to zero as $y$ goes to zero. Hence, so it does 
$\pi_2(F_2(x,y)-(x,y))$.}

\label{dibujotakinfspikerot}
\end{center}\end{figure}

Due to the geometry of the foliation and the definition of the vector field, we
have that horizontal lines which are close to $\cR$ get mapped in curves which
remain close to $\cR$ (see Figure \ref{dibujotakinfspikerot}).  Formally
speaking we have that the family of maps $\{\Delta_y:\R^+\rightarrow
\R^+\}_{y\in\R^+}$ defined by $\Delta_y(x)=\pi_2(F_2(x,y))$, tends uniformly to
the constant zero as $y$ goes to zero.

Thus we can consider $y_0\in (0,\frac{1}{3}]$ such that $F_2(\cR_{y_0})\subset
(0,\frac{1}{3}]$. This implies that we can now define a $T$-invariant bijective
function $W:\R\times[y_0,\frac{2}{3}]\rightarrow
\overline{\mathcal{U}^+(F_2(\{(x,y_0)\mid
  x\in\R\}))}\cap\mathcal{U}^-(\cR_{\frac{2}{3}})$ given by
$W(x,y)=H_{s(y)}\circ F_1(x,y)$, where $H_s$ is the time $s$ map of the flow
associated to $X$, and $s:[y_0,\frac{2}{3}]\rightarrow [0,1]$ is a monotone
decreasing continuous function from $1$ to $0$ such that $H_{s(y)}$ defines an
injective function in each set of the form $l_p\cap [y_0,\frac{2}{3}]$, $p\in\R$
(we omit the details for the construction of $s$).
  
Then $W$ leaves $\mathcal{T}$ invariant and verifies
$W\mid_{\cR_{y_0}}=F_2\mid_{\cR_{y_0}}$ and
$W\mid_{\cR_{\frac{2}{3}}}=F_1\mid_{\cR_{\frac{2}{3}}}$.  Consequently, we can
define a $T$-invariant homeomorphism $F:\R^2\rightarrow \R^2$ as follows.
\begin{equation*}
F_{g,\alpha}(x,y)= \ \left\{ \begin{array}{lcl}

F_1(x,y) & \textrm{if} & (x,y)\in\mathcal{U}^+(\cR_{\frac{2}{3}})\\

W(x,y) & \textrm{if} &(x,y)\in\overline{\mathcal{U}^-(\cR_{\frac{2}{3}})\cap\mathcal{U}^+(\cR_{y_0})}\\

F_2(x,y)  & \textrm{if} &(x,y)\in\mathcal{U}^+(\cR)\cap \mathcal{U}^-(\cR_{y_0})\\

x+\alpha  & \textrm{if} & (x,y)\in \overline{\mathcal{U}^-(\cR)}
\end{array}\right.
\end{equation*}

In order to see that $F$ is a homeomorphism, we have to check that $F$ is
continuous for a point $(x,0)\in\R^2$. However, due to construction, for any
sequence $(w_n)_{n\in\N}$ converging to $(x,0)$ we have that
$F_2(w_n)-w_n=(\alpha,d_n)$ with $\lim_n d_n=0$.  Thus $\lim_n
F_2(w_n)=(x+\alpha,0)$.

Hence, $F_{g,\alpha}$ is a $T$-invariant homeomorphism of the upper half-plane,
which can easily be extended to all of $\R^2$ and thus defines a homeomorphism
$f_{g,\alpha}:\A\to\A$. Furthermore $f_{g,\alpha}$ leaves invariant the
essential thin annular continuum given by
$A_{g,\alpha}:=\pi(\cR_0\cup\cT)=\Omega(g)\times\{1\}$, which has at least one
infinite spike and whose circloid
$C_{A_{g,\alpha}}=\pi(\cR_0)=\kreis\times\{0\}$ is compactly generated.

Finally, it follows from the definition of $f$ that points in $A\cap\pi(\cR_1)$
have a unique rotation vector $(\rho(G),0)$, and points in $C_A$ have rotation
vector $(\alpha,0)$. Hence $\rho_A(F)\supset I=\textrm{conv}(\alpha,\rho(g))$.
Furthermore, given any point $z\in A\setminus (C_A\cup\pi(\cR_1))$, we have by
construction that $\pi_1(F^n(z)-z)\in
\textrm{conv}(n\cdot\alpha,[G^n(h(z))-h(z)])$. This implies that $\rho_A(F)=I$.

\end{subsection}

\bibliography{circloids,dynamics} \bibliographystyle{unsrt}

\end{document}

\subsection{Irrational combinatorics.}
Given three points $x,y,z\in\kreis$ we denote by $[x,z]$ the closed interval
from $x$ to $z$ in positive direction (counterclockwise for the identification
$\kreis\to\Seins,\ x\mapsto e^{2\pi ix}$) and write $x\leq y\leq z$ if
$y\in[x,z]$. We use analogous notation for open intervals and write $x<y<z$ if
$y\in (x,z)$. 

Using this circular order, we say a sequence \nfolge{x_n} in \kreis\ has {\em
  irrational combinatorics} if there exists $\rho\in\R\smin\Q$ such that
\nfolge{x_n} has the same combinatorial order as the orbits of the irrational
rotation $R_\rho : \kreis\selfmap,\ x\mapsto x+\rho\bmod 1$. This means that for
a sequence $y_n=R^n_\rho(y_0)$, with arbitrary $y_0\in\kreis$, we have 
\[
x_k \leq x_m \leq x_n \quad \equi \quad y_k  \leq y_m \leq y_n 
\]
for all $k,m,n\in\Z$.